\documentclass[a4paper, 12pt, twoside]{article} 

    \usepackage{verbatim}
    \usepackage[left=3cm,right=2.5cm,bottom=3.5cm,top=2.5cm]{geometry} % Seitenrändereinstellungen.
    \usepackage{amsfonts}
    \usepackage[utf8]{inputenc} % Kodierung
    \usepackage[english]{babel} % Sprache
    \usepackage{amssymb}        % Mathematische Symbole wie z.b ganze Zahlen, reelle Zahlen etc.
    \usepackage{amsmath}        % Um diverse mathematische Symbole nutzen zu können.
    \usepackage{amsthm}         % Um Definitionen, Theoreme, Bemerkungen, Beispiele machen zu können.
    \usepackage{mathtools}      % Dieses Paket liefert nützliche Werkzeuge, z.B "defined as equal" - sign uvm.
    \usepackage{enumitem}       % Um individuelle Listen zu bauen.
    \usepackage{xcolor}         % Um farbigen Text machen zu können. Diesen kann ich nutzen um wichtige persönliche Notizen hervorzuheben.
    \usepackage{fancyhdr}       % Um saubere Kopf und Fußzeilen sowie Seitenzahlen zu erzeugen.
    \usepackage{setspace}       % Damit kann ich Leerzeilen im Dokument einfügen. Dies ist insbesondere beim Inhaltsverzeichnis nötig, da dieses sonst zu gequetscht aussieht.
    \usepackage{bbm}            % Damit ich einen Indikator machen kann.
    \usepackage{hyperref}       % Dieses Paket sorgt dafür, dass ich die Überschriften des Inhaltsverzeichnisses als Links nutzen kann.
    \usepackage{graphicx} 
    \usepackage{float}
    \usepackage{multirow}
    \usepackage{diagbox}
    \usepackage{hhline}
    \usepackage{caption}        % Damit ich bei figures captions schreiben kann, ohne figure Nummerierung.

    \usepackage{algorithm}      % Um Algorithmen-Umgebungen zu definieren.

    \usepackage{subfig}

     \usepackage[authoryear]{natbib}
	\newcommand{\ts}{\thinspace} % Damit ich nicht so viel schreiben muss, wenn ich kleine Leerzeilen hinzufügen möchte. Was ich oft tue bei Mengendefinitionen, da mir der Platz dort zu klein ist.
	
		\newtheoremstyle{Format1}
		{\topsep}   % ABOVESPACE
		{\topsep}   % BELOWSPACE
		{\normalfont}  % BODYFONT
		{0pt}       % INDENT (empty value is the same as 0pt)  % Das rückt den Kopf nach rechts ein
		{\bfseries} % HEADFONT
		{\newline}  % HEADPUNCT
		{5pt plus 1pt minus 1pt} % HEADSPACE
		{}          % CUSTOM-HEAD-SPEC 

		\newtheorem{Def}{Definition}[section]       % Definition
            \newtheorem{Lem}[Def]{Lemma}                % Lemma
            \newtheorem{Thm}[Def]{Theorem}              % Theorem
            \newtheorem{rem}[Def]{Remark}               % Remark

        \DeclarePairedDelimiter\abs{\lvert}{\rvert}%
        \DeclarePairedDelimiter\norm{\lVert}{\rVert}%
        
        \makeatletter
        \let\oldabs\abs
        \def\abs{\@ifstar{\oldabs}{\oldabs*}}
        \let\oldnorm\norm
        \def\norm{\@ifstar{\oldnorm}{\oldnorm*}}
        \makeatother
        
        \hypersetup{
            colorlinks,
            citecolor=black,
            filecolor=black,
            linkcolor=black,
            urlcolor=black
        }
        
        \title{Comparing regression curves - an $L^1$-point of view}
\author{
  {\normalsize Patrick Bastian, Holger Dette,  Lukas Koletzko} \\
{\normalsize  Ruhr-Universit\"at Bochum} \\
{\normalsize  Fakult\"at f\"ur Mathematik} \\
{\normalsize  44780 Bochum, Germany}
\and 
{\normalsize  Kathrin Möllenhoff} \\
{\normalsize  Heinrich-Heine-Universit\"at D\"usseldorf} \\
{\normalsize  Mathematisches Institut} \\
{\normalsize  40225 D\"usseldorf, Germany}
}

\date{}

\begin{document}
     
        \maketitle

        \begin{abstract}
        In this paper we  compare two regression curves by measuring their difference by the area between the two curves, represented by their $L^1$-distance.
        We develop  asymptotic confidence intervals for this measure and statistical tests to investigate the similarity/equivalence of the two curves. 
        Bootstrap methodology specifically designed for equivalence testing is developed to obtain procedures with good finite sample properties and its consistency is rigorously proved.
        The finite sample properties are investigated by means of a small simulation study.
        \end{abstract}

\medskip
  \noindent
  Keywords:   equivalence testing, comparison of curves, bootstrap, directional Hadamard differentiability

\noindent AMS Subject classification:  62F40;  62E20; 62P10

        	\section{Introduction}
        	%\label{sec1} 
  \def\theequation{1.\arabic{equation}}	

The comparison of  regression curves is
a common problem in applied regression analysis. Usually these curves correspond to the means of a control  and a treatment  outcome where the predictor variable is an adjustable parameter, such as the  time or  a dose level, 
and an important question is  whether the difference between  the two curves is practically irrelevant. Borrowing ideas from testing for  bioequivalence in population pharmacokinetics \citep[see, for example][]{chowliu1992,haustepig2007}
numerous authors  have  addressed the problem of establishing 
the similarity between two regression functions by testing  
hypotheses of the form 
\begin{align} \label{hd11}
H_0: d \geq \epsilon  \text{~~~ versus  ~~~}
H_1: d < \epsilon ,
\end{align}
where $d$ is a distance between the curves (which vanishes, if they are identical) and $\epsilon $ is a threshold, which defines  when the difference between the two  curves is considered as practically irrelevant.

Hypotheses of this type  
             have found considerable interest in the literature. 
             In many applications  the sample sizes are small  such that  nonparametric approaches are prohibitive and the  relations  between predictor and response  for the two groups are modelled by   nonlinear regression models with low dimensional parameters. 
              \cite{liubrehaywynn2009}
proposed tests for comparing linear models, while \cite{gsteiger2011}
suggested a bootstrap test for  nonlinear models.
The tests in these papers  are  based on the  intersection-union principle \citep[see, for example,][]{berger1982} and
a  confidence band for the difference of the two regression models
\citep[see also][]{liuhaywyn2007}.
\cite{dette2018equivalence} pointed out that, by construction,  these tests  are very conservative and proposed an alternative approach, which  has been successfully applied by 
\cite{mollenhoff2022testing}, 
among others. A common feature of the publications in this field 
consists in the fact that all methods use the maximal deviation 
$d_{\infty} :=  \sup_{x} |{m_{1}(x)-m_{2}(x)}   |
$
for the comparison of the curves, say 
$m_1$ and $m_2$ (here $x$ denotes the predictor). While this metric has some attractive features such as good interpretability 
and a simple view of large distances, it is often
too restrictive as one  is interested in a ``worst case'' scenario on a local scale. More precisely, if one uses the metric of maximum deviation between the two curves 
one is not  able to decide for  similarity 
if the curves are very similar for most points $x \in {\cal X}$, except for a ''small'' region. 
In such cases  an ''average'' 
\begin{align*}
d_{1}  := \int  | {m_{1}(x)-m_{2}(x)} |  dx 
\end{align*}
measuring the area between the curves might have advantages
and  in this paper we develop statistical inference tools if two regression curves are compared on the basis of  this $L^1$-distance.
%Of course, in general  the choice of the concrete metric depends on the specific application and is therefore problem specific.

Our interest for this distance stems from the fact that the area under the curve (AUC) is  a quite  popular measure in biostatistics. For example, according to current guidelines by regulation authorities 
both in the US and the EU  bioequivalence between a reference  and a test  product 
is to be assessed based on the comparison of their respective area under the time-concentration 
curves  and  their  maximal concentrations (see \cite{food2003guidance,ema} for details on that).
For  analyses of clinical trials with a time-to-event outcome \cite{Royston2013} and \cite{Zachary2019} proposed the 
restricted mean survival time (RMST) as a possible alternative tool to the commonly used  hazard
ratio to estimate the treatment effect. 
By introducing this measure the authors address the well-known issue of assuming proportional hazards between different arms of the trial- an assumption, which is often unrealistic or obviously violated, but, however, only rarely assessed in practice (see \cite {jachno2019non} for a recent overview). 
The RMST, that is the area under the survival curve up to a specific time point, comes along without any assumption on the shape of the hazard ratio by calculating a treatment effect as the difference in RMST. 
We also refer to \cite{coxczan2016} who investigated the $L^1$-distance for comparing survival distributions.
%Moreover,  the  area under the receiver operating characteristic (ROC) curves are common tools for assessing how well  clinical risk prediction models distinguish between patients with and without a health outcome of interest \citep[see][among others]{pepeetal2013,helleretal2016}. 
Moreover,  an important performance metric for assessing how well  clinical risk prediction models distinguish between patients with and without a health outcome of interest is the area under receiver operating characteristic (ROC) curves (short AUROC), a tool of huge practical interest \citep[see][among others]{pepeetal2013,helleretal2016}. 
Also apart from medical questions, the AUROC is a common used tool in machine learning, arising whenever classifiers are evaluated or compared to each other regarding their power, see, for example \cite{bradley1997use}.

As pointed out by  \cite{coxczan2016} the choice of $L^1$-distance for the comparison of curves poses several mathematical challenges, which are caused by the fact that the mapping $f \to \int_{\cal X} |f(x)| dx $  is in general not (Hadamard-)differentiable. 
These authors considered the distance
$\int | S_1f_2 (x)  - S_2f_1 (x) | dx 
$ and  
restricted themselves to the situation where $S_1f_2 \geq S_2f_1$
to solve the differentiability  problem 
(in this case  the absolute value in the integral  can be omitted). In the context of  
testing for the equivalence of multinomial distributions 
\cite{ostr2017}  proposed  to
use a smooth approximation of the $L^1$-norm
 to avoid the differentiability problem.

Our approach for investigating the similarity between the regression functions $m_1$ and $m_2$ avoids such approximations
and does not require that one regression function is larger than the other one. It  is based on a direct  estimate of the distance $d_{1}$ whose asymptotic properties are investigated  in Section  
 \ref{sec2}. The results are used for the construction  of asymptotic confidence intervals for $d_1$ and a corresponding test for the hypotheses  \eqref{hd11} by duality principles. In order to 
obtain less conservative tests with good properties for small sample sizes  we propose a constrained bootstrap test in  Section \ref{sec3}. Section \ref{sec4} is devoted to a small simulation study illustrating good finite sample properties of the bootstrap confidence intervals and tests. Finally, all proofs and technical details are given in Section \ref{appendix}.

    	    \section{Comparing the area between the curves}
    	    \label{sec2} 
  \def\theequation{2.\arabic{equation}}	
  \setcounter{equation}{0}
  
        	We consider two independent samples of $n_1$ and  $n_2$ observations. In each group  ($\ell=1,2$) there exist $k_\ell$ different covariates, say  $x_{\ell,1},  \ldots  , x_{\ell,k_{\ell}}$,  such that  at each covariate $x_{\ell,i}$ $n_{\ell,i}$
        	independent identically distributed observations 
        	$\{   Y_{\ell,i,j}  : j=1,\ldots , n_{\ell, i} \} $ are available such that $n_\ell = \sum_{i=1}^{k_{\ell} } n_{\ell,i} $  ($\ell=1,2$). The total sample size is denoted by   $n= n_{1} + n_{2}$.
        	We assume  that the covariates vary  in a set  $\mathcal{X} \subset \mathbb{R}^{d}$
        	(for some $d \in \mathbb{N}$), and that the 
       relation between the covariates  and responses can be represented by a non-linear regression model
       of the form
        	\begin{eqnarray}
        	    \label{l1}
        	    Y_{\ell,i,j} = m_{\ell}(x_{\ell,i},\beta_{\ell}) + \eta_{\ell,i,j}, \ \ j = 1, \ldots ,n_{\ell,i}, \ i = 1, \ldots ,k_{\ell},
        	\end{eqnarray}
        	    where  $m_{\ell}(\cdot,\beta_{\ell}) \in \ell^{\infty}(\mathcal{X})$  is  the regression function with parameters $\beta_{\ell} \in \mathbb{R}^{p_{\ell}}$, $p_{\ell} \in \mathbb{N}$ and  $\ell^{\infty}(\mathcal{X})$ denotes the space of bounded real-valued functions
        	    $f:{\cal X} \to \mathbb{R}$. In model 
        	    \eqref{l1} the quantities 
        	$\{ \eta_{\ell,i,j} : 
        	 j = 1, \ldots ,n_{\ell,i}, \ i = 1, \ldots ,k_{\ell} \} $ denote independent identically distributed random variables with mean $ 0$  and variance $\sigma_{\ell}^2 >0 $ $, \ \ell =1,2$.
    
        	 In this paper we are interested in the similarity between the regression functions $m_1$ and $m_2$, where the distance between the two functions is measured by the $L^1$-norm. More precisely, 
          we consider  the similarity parameter
        	\begin{eqnarray}
        	\label{l2}
        	    d_{1} = d_{1}(\beta_{1},\beta_{2}) \coloneqq \int_{\mathcal{X}} \abs{m_{1}(x,\beta_{1})-m_{2}(x,\beta_{2})} \ts dx  ~, 
        	\end{eqnarray}
        	and develop confidence intervals for $d_1$ and statistical tests for the hypotheses 
        	\begin{eqnarray}
        	\label{l3}
        	    H_{0}: d_{1}(\beta_{1},\beta_{2}) \geq \epsilon \quad \text{ vs. } \quad H_{1}:  d_{1}(\beta_{1},\beta_{2}) < \epsilon ,
        	\end{eqnarray}
            where $ \epsilon >0 $ is a pre-specified constant.
            Note that the rejection of $H_0$ in \eqref{l3} allows to decide at a controlled type I  error  that the area between the two curves is smaller than a given threshold.

        	As pointed out in the introduction 
         the choice of $L^1$-distance for the comparison of curves poses several mathematical challenges, which are caused by the fact that in contrast to the $\sup$- and the $L^2$-norm, the mapping $f \to \int_{\cal X} |f(x)| dx $ from $l^\infty ({\cal X} )$ onto
        	$\mathbb{R}$ is in general not (Hadamard-)differentiable. 
         % doppelt: (here $\ell^{\infty}(\mathcal{X})$      denotes the set of all bounded functions $f: {\cal X}  \to \mathbb{R}$).       
Our approach for investigating the similarity between the regression functions $m_1$ and $m_2$   is based on a direct  estimate of the distance $d_{1}(\beta_{1},\beta_{2})$. To be precise, let 
  $\hat \beta_1$ and $\hat{\beta}_{2}$ denote appropriate  estimates of  the parameters in model
  $m_1(\cdot, \beta_1) $ and   $m_2(\cdot, \beta_2) $
  obtained from the samples 
  $\{Y_{1,i,j}  | j = 1, \ldots ,n_{1,i}, \ i = 1, \ldots ,k_{1}, \}$
  and $\{Y_{2,i,j}  | j = 1, \ldots ,n_{2,i}, \ i = 1, \ldots ,k_{2}, \}$, respectively (precise assumptions on the properties of these estimates are given in the appendix). The estimate of the area between the curves $d_1$ in \eqref{l2} is then defined by 
  \begin{eqnarray}
  \label{l4}
      \hat d_{1}=d_{1}(\hat{\beta}_{1},\hat{\beta}_{2}) = \int_{\mathcal{X}} | m_{1}(x,\hat{\beta}_{1}) - m_{2}(x,\hat{\beta}_{2}) |  dx,
  \end{eqnarray}
 and we will show 
  that the normalized statistic  $\sqrt{n}(\hat d_{1} - d_{1}) $  converges weakly to a  random variable $ T$
     with non-degenerate distribution.
For this prupose we  introduce the set 
        		\begin{eqnarray}
        		    \label{h1} 
        		{\cal N} := \{ x \in {\cal X} |~ m_{1} (x, \beta_1) - m_{2} (x, \beta_2)= 0 \}   
        				\end{eqnarray}
        				as the set of points, where the two regression functions coincide. 
           Throughout this paper the symbol $	\xrightarrow{d}$ 
 means weak convergence in distribution.

        	% Theorem 1
        	\begin{Thm}
        	\label{thm1}
        		\textit{Suppose that Assumption 1-7 in the appendix are satisfied, in particular $n_1,n_2 \to \infty $ such that $ {n/n_1} \rightarrow \kappa  $  $\in (1, \infty) $ 
          and  $n_{\ell, i}/n_\ell \rightarrow \zeta_{\ell, i}
          \in (0,1) $ for $i=1, \ldots ,k_\ell$,  $\ell=1,2$. The statistic $\hat d_{1}$
        		defined in \eqref{l4} satisfies
        		\begin{eqnarray}
        		\label{h4}
        		\sqrt{n}(\hat d_{1}-d_{1}) & %\sqrt{n}\int_{\mathcal{X}} \abs{m_{1}(x,\hat{\beta_{1}})-m_{2}(x,\hat{\beta_{2}})} - \abs{ m_{1}(x,\beta_{1})-m_{2}(x,\beta_{2})} \ dx
        		\xrightarrow{d} T: = \int_{
        		{\cal N}^c} \operatorname{sgn}(m_{1}(x,\beta_{1}) - m_{2}(x,\beta_{2})) \ G(x) \ dx  + \int_{
        		{\cal N} } \abs{G(x)} \ dx \ , ~~~~~~~~
         		\end{eqnarray}
         		where
         		$\{G(x)\}_{x \in \mathcal{X}}$ is a centred Gaussian process in $\ell^{\infty}(\mathcal{X})$ defined  by 
         		\begin{eqnarray*}
         %		\label{l5}
         		    G(x) = \Big (\dfrac{\partial}{\partial b_{1}}m_{1}(x,b_{1})\Big  |_{b_{1}=\beta_{1}}\Big )^{\top}\sqrt{\kappa }\Sigma_{1}^{-1/2}Z_{1}-\Big (\dfrac{\partial}{\partial b_{2}}m_{2}(x,b_{2}) \Big |_{b_{2}=\beta_{2}}\Big )^{\top}\sqrt{\dfrac{\kappa }{\kappa  - 1}}\Sigma_{2}^{-1/2}Z_{2},
         		\end{eqnarray*}
         		 $Z_{1}$ and $Z_{2}$ are  $p_1$ and $p_2$-dimensional standard normal distributed random variables, respectively, 
         	and the matrices $\Sigma_1$ and $\Sigma_2$ are defined by 
         	\begin{eqnarray*}
         	%    \label{h9}
         	    \Sigma_\ell = \dfrac{1}{\sigma_{\ell}^2} \sum_{i=1}^{k_{\ell}} \zeta_{\ell, i} \Big(\dfrac{\partial}{\partial b_{\ell}} m_{\ell}(x_{\ell,i},b_{\ell})\Big|_{b_{\ell}=\beta_{\ell}}  \Big)\ts \Big( \dfrac{\partial}{\partial b_{\ell}} m_{\ell}(x_{\ell,i}, b_{\ell})\Big |_{b_{\ell}=\beta_{\ell}}\Big)^\top ~~~~(l=1,2). ~~~~
         	\end{eqnarray*}
         	}
        	\end{Thm}	
        	If  the distribution of $T$ in \eqref{h4}  would be known and $q_{\alpha}$  denotes  the  corresponding  $\alpha$-quantile, it follows from   Theorem \ref{thm1}  that  the  
         coverage probability of the  ``oracle'' confidence interval
        $ [0 , \hat d_1  - \frac{q_{\alpha}}    {\sqrt{n}} )$ 
        for the parameter $d_1$ converges with increasing sample size  to $ 1 - \alpha$. 
        Similarly, a simple calculation shows that the test, which
        rejects the null hypothesis, whenever
      $  	
        	    \label{h2}
        	    \hat d_1 < \epsilon + \frac{q_{\alpha}}{\sqrt{n}},
       $ 	
        is a consistent asymptotic level $\alpha$ test for the hypotheses
        	\eqref{l3}.

 However, the distribution of the limiting random variable $T$  in \eqref{h4} is not easily accessible, 
      because it depends on certain unknown nuisance parameters, in particular on the set ${\cal N} $ of points, where the two functions $m_1$ and $m_2$ coincide. The unknown covariance matrices
      $\Sigma_1$ and $\Sigma_2$, which essentially define the Gaussian process $\{G(x)\}_{x \in \mathcal{X}}$,  can be estimated in a rather straightforward manner by
      \begin{eqnarray}
          \label{h6}
          \hat{\Sigma}_\ell  = \dfrac{1}{\hat{\sigma}_{\ell}^2} \sum_{i=1}^{k_{\ell}}{ n_{\ell, i} \over n_{\ell}} \Big(\dfrac{\partial}{\partial b_{\ell}} {m}_{\ell}(x_{\ell,i},b_{\ell})\Big|_{b_{\ell}=\hat{\beta}_{\ell}} \Big) \ts \Big( \dfrac{\partial}{\partial b_{\ell}} {m}_{\ell}(x_{\ell,i}, b_{\ell})\Big|_{b_{\ell}=\hat{\beta}_{\ell}}\Big)^\top~,
      \end{eqnarray}
where 
$\hat\sigma_{1}^2$ and $\hat\sigma_{2}^2$ are consistent estimators of the variances $\sigma_{1}^2$ and $\sigma_{2}^2$, respectively.  
      On the other hand, the estimation of the set ${\cal N} $ is more difficult. For a constant $c >0$ we  define an estimate by 
      \begin{eqnarray}
          \label{h55}
          \hat{\cal N} = \Big \{ x \in {\cal X}~ \Big |
          ~ \big |  m_{1}(x,\hat{\beta_{1}}) - m_{2}(x,\hat{\beta_{2}}) \big |  <  c \sqrt{\frac{\log n}{n}} \Big \}~.
      \end{eqnarray}
        Consequently, let $\hat G$ denote the process $G$ introduced in Theorem \ref{thm1}, where the 
        parameter $\beta_\ell$ and the 
        matrix  $\Sigma_\ell$  have  been replaced by  their estimates  $\hat \beta_l$
and    $ \hat{\Sigma}_\ell $, respectively  and $\kappa$ by $n/n_1$. Then we define the random variable
    \begin{eqnarray}
   	\label{h7}
        		\hat{T}: = \int_{
        	\hat{\cal N}^c } \operatorname{sgn}(m_{1}(x,\hat \beta_{1}) - m_{2}(x,\hat \beta_{2})) \ \hat   G(x) \ dx  + \int_{
        	\hat{\cal N} }  | \hat G(x) |  \ dx \ ,
          		\end{eqnarray}  
         		and denote by $\hat q_{0,\alpha}$ the  $\alpha$-quantile  of the corresponding distribution conditional on $\hat \beta_1, \hat \beta_2$, which can easily be simulated.  We now define 
         		a confidence interval 
         		    \begin{eqnarray}
               \label{conf}
          \hat I_n := \Big [0, \hat{d}_1 - \dfrac{\hat q_{0,\alpha}}{\sqrt{n}} \Big ),
      \end{eqnarray}
         		for the $L^1$-distance  $d_1$  and, using the duality between confidence intervals and statistical tests
         \citep[see, for example][]{aitchison1964},   we propose 
         to reject the null hypothesis in
       \eqref{l3}, whenever
       $\hat  I_n \subset [0, \epsilon )$, which is equivalent to 
        \begin{eqnarray}
        	    \label{h88}
        	    \hat d_1 < \epsilon + \dfrac{\hat q_{0,\alpha}}{\sqrt{n}}~.
    \end{eqnarray} 
  
    We then obtain the following theorem on the significance level and the power of the above procedure.

      \begin{Thm}  \label{thm22} 
      {\it If the assumptions of Theorem \ref{thm1} are satisfied, then  \eqref{conf} defines an asymptotic confidence interval for the quantity $d_1$, that is
      \begin{align*}
          \lim_{n \rightarrow \infty} \mathbb{P}\big( d_1 \in \hat{I}_n \big) =  \lim_{n \rightarrow \infty} \mathbb{P}\Big( d_1 \in \Big [0, \hat{d}_1 - \dfrac{\hat q_{0, \alpha}}{\sqrt{n}} \Big ) \Big)=1-\alpha
      \end{align*}
      Moreover, the test defined in \eqref{h88} is a consistent and asymptotic level $\alpha$-test, that is

\begin{itemize}
      \item[(1)]  If  $d_1 \geq \epsilon $, ~ then ~
$
          \limsup_{n \rightarrow \infty} \mathbb{P}\big(\hat d_1 < \epsilon + \frac{\hat q_{0,\alpha}}{\sqrt{n}} \big) \leq \alpha 
          $
    \item[(2)]  If  $d_1<  \epsilon $, ~ then   ~      
  $
          \liminf_{n \rightarrow \infty} \mathbb{P}\big(\hat d_1 < \epsilon + \frac{\hat q_{0,\alpha}}{\sqrt{n}} \big)  = 1.
$
      \end{itemize}
      
       }
      \end{Thm}

        \begin{rem} 
        {\rm 
        %\label{r2}
      A two-sided (asymptotic) confidence   interval 
      for   $d_1$ is given by 
 $    
         [ \hat d_1 - \frac{\hat q_{0, 1-\alpha/2}}{\sqrt{n}}, \ \hat d_1 - \frac{\hat q_{0, \alpha/2}}{\sqrt{n}}].
 $    }
      \end{rem}

        	\section{Bootstrap methodology}
        	\label{sec3}
        	  \def\theequation{3.\arabic{equation}}
\setcounter{equation}{0}

         The confidence interval  and test proposed in Section \ref{sec2} require the precise estimation of the set ${\cal N} $  in \eqref{h1}, which might be difficult for small sample sizes. In order to address this  problem and to obtain also a better approximation of the nominal level we propose a parametric bootstrap approach. While this is quite standard for the construction of confidence intervals         
         we will develop and investigate a novel constrained bootstrap approach for testing the hypotheses  \eqref{l3}.
         This approach  will result in more powerful tests compared to the confidence interval approach (see Section \ref{sec4} below).
     We describe  the construction of a bootstrap confidence interval
        in Algorithm \ref{alg1}.
\begin{algorithm}
\caption{Bootstrap confidence interval}
\label{alg1}
\begin{itemize}
    \item[(1)] Calculate  the estimates $\hat \beta_1$ and $\hat \beta_2$ (for each group).
    \item[(2)] For  $\ell = 1, 2$, $i=1, \ldots , k_\ell $, $j=1, \ldots , n_{\ell,i}$ generate bootstrap data from the model
       \begin{eqnarray}
        \label{h5c}
        Y_{\ell,i,j}^{*} = m_{\ell}(x_{\ell,i},{\hat{\beta_{\ell}}}) + \eta_{\ell,i,j}^{*}, 
       \end{eqnarray} 
    where the errors $\eta_{\ell,i,j}^{*}$ are independent centered normal  distributed with variance $\hat\sigma_{\ell}^2$. %denotes the  variance estimate for group $\ell =1,2$ defined in \eqref{h12}.
    \item[(3)] Let $\hat\beta_{1}^{*}$ and  $\hat\beta_{2}^{*}$ denote the estimators of $\beta_1$ and $\beta_2$ from the bootstrap data in \eqref{h5c}. Calculate the bootstrap test statistic $\hat d_{1}^{*} := d_{1}(\hat\beta_{1}^{*},\hat\beta_{2}^{*})$ and define $\hat{q}_{1-\alpha,0}^*$ as the $(1-\alpha)$-quantile of the distribution of $\hat d_{1}^{*}$.
    \item[(4)] The bootstrap confidence interval is defined by 
       \begin{eqnarray}
       \label{h7c}
        \hat I_n^* := \big   [0,    \hat{q}_{1-\alpha,0 }^{*} \big ). 
       \end{eqnarray}
\end{itemize}
\end{algorithm}
            
As in Section \ref{sec2} the duality between confidence intervals and statistical tests \citep[see][]{aitchison1964} yields a test for the hypotheses \eqref{l3}, which rejects the null hypothesis, whenever
$ \hat I_n^* \subset [0, \epsilon )$, that is 
\begin{eqnarray} \label{hd21}
     \hat{q}_{1-\alpha,0 }^{*} < \epsilon
\end{eqnarray}
The following result shows that this adhoc approach yields a consistent asymptotic level $\alpha $ test.

\begin{Thm}  \label{thm23} 

      {\it Let the assumptions of Theorem \ref{thm1} be  satisfied
      and assume that $\lambda(\mathcal{N})=0$.  The  interval  $ \hat{I}^*_n$ in \eqref{h7c} defines an asymptotic confidence interval for the quantity $d_1$, i.e.
      \begin{align*}
          \lim_{n \rightarrow \infty} \mathbb{P}\big( d_1 \in \hat{I}^*_n \big) 
          %=  \lim_{n \rightarrow \infty} \mathbb{P}\Big( d_1 \in \Big [0,  \hat{q}_{1-\alpha,0 }^{*} \Big ) \Big)
          =1-\alpha.
      \end{align*}
     Moreover, the test defined in \eqref{hd21} is a consistent and asymptotic level $\alpha$-test, that is   

\begin{itemize}
      \item[(1)]  If  $d_1 \geq \epsilon $ then  ~
$
          \limsup_{n \rightarrow \infty} \mathbb{P}\big( \hat{q}_{1-\alpha,0 }^{*} < \epsilon\big) \leq \alpha 
          $
    \item[(2)]  If  $d_1<  \epsilon $ then       ~  $
          \liminf_{n \rightarrow \infty} \mathbb{P}\big(\hat{q}_{1-\alpha,0 }^{*} < \epsilon\big)  = 1.
$
      \end{itemize}
      }
\end{Thm}

The finite sample properties of the confidence interval \eqref{h7c} and the test \eqref{hd21} will be investigated in Section \ref{sec4}. In particular it will be demonstrated that, by its construction,  the  test \eqref{hd21}  is rather conservative and not very powerful.
Therefore we propose as an alternative a constrained 
bootstrap test for the hypotheses \eqref{l3}, which addresses the specific structure of the composite hypotheses \eqref{l3}. The pseudo code for this test is summarized in Algorithm \ref{alg2}.

\begin{algorithm}[t]
\caption{(Constrained parametric bootstrap test}
\label{alg2}
\begin{itemize}
       \item[(1)] Calculate the test statistic $\hat d_{1} := d_{1}(\hat{\beta_{1}},\hat{\beta_{2}})$ defined in \eqref{l4}.
       
       \item[(2)] Calculate the  estimators $\tilde{\beta_{1}}$  and $ \tilde{\beta_{2}}$  of the parameters $\beta_{1}$ and $\beta_{2}$, respectively,  under the additional constraint that $ d_{1}(\beta_{1},\beta_{2}) = \epsilon$. Define 
       	\begin{eqnarray}
       	\label{h8}
       	\hat{\hat{\beta_{\ell}}} = 
             \begin{cases}
       			\hat{\beta_{\ell}}, \ \hat d_{1} \geq \epsilon \\
       			\tilde{\beta_{\ell}}, \ \hat d_{1} < \epsilon.
             \end{cases}, \ \ell = 1, 2.
             \end{eqnarray}
       \item[(3)] 
       			%Repeat the following two steps $B$ times:
       		%	\begin{itemize}
       		%	    \item[3.1] 
       		For  $\ell = 1, 2$, $i=1, \ldots , k_\ell $, $j=1, \ldots , n_{\ell,i}$ generate bootstrap data from the model
       		\begin{eqnarray}
       		    \label{h5}
       	  			Y_{\ell,i,j}^{*} = m_{\ell}(x_{\ell,i},\hat{\hat{\beta_{\ell}}}) + \eta_{\ell,i,j}^{*}, 
       				\end{eqnarray} 
       			where 
       			the errors $\eta_{\ell,i,j}^{*}$ are independent centered  normal  distributed  with variance 
       			$
       			\hat\sigma_{\ell}^2$.
          % denotes the  variance estimate for group $\ell =1,2$ defined in \eqref{h12}.	
  \item[(4)] Let $\hat\beta_{1}^{*}$ and  $\hat\beta_{2}^{*}$ denote the  estimators of $\beta_1$ and $\beta_2$ from the bootstrap data
  in \eqref{h5}.
       			    Calculate the bootstrap test statistic $\hat d_{1}^{*} := d_{1}(\hat\beta_{1}^{*},\hat\beta_{2}^{*}) $
       			    and define $\hat{q}_{\alpha,1}^*$
       			    as the $\alpha$-quantile of the distribution of $\hat d_{1}^{*}$.
       % 			\end{itemize}
       			\item[(5)] The null hypothesis in \eqref{l3} is rejected, whenever
       	\begin{eqnarray}
       	\label{h7}
        \hat d_{1} < \hat{q}_{\alpha,1}^{*}. 
       	      	\end{eqnarray}
       		
       		\end{itemize}
   \end{algorithm}

   \begin{rem}  ~~\\
   {\rm 
   (a)  		Note that,  by   definition  \eqref{h8},
   $\hat{\hat d} _{1} := d_{1}(\hat{\hat{\beta_{1}}}\hat{\hat{\beta_{2}}}) \geq \epsilon$. Therefore, Algorithm \ref{alg2}
  generates in step (3)  bootstrap data under the null hypothesis. 
   \smallskip
   
   (b) In practice the quantile can be estimated with arbitrary precision  generating bootstrap replicates $\hat{d}_{1}^{*(1)},  \ldots  , \hat{d}_{1}^{*(B)}$  as described in step (3) and (4) and 
   calculating the empirical $\alpha$-quantile, say $\hat{q}_{\alpha,1}^{(B)}$, from this sample. 

   (c)         The following result shows that this parametric bootstrap test has asymptotic level $\alpha$ and is consistent if   the set ${\cal N}$
        defined in \eqref{h1} has Lebesgue measure $0$.
   }
\end{rem}

       	\begin{Thm}
       	\label{thm2}
       		\textit{Let Assumption 1-7 in the appendix 
       		be satisfied and  assume that
       		the set ${\cal N}$
        defined in \eqref{h1} has Lebesgue measure zero. % If $\alpha < 0$ (lk: Ich denke das ist falsch.)
       Further assume that the $\alpha$-quantile  $q_\alpha$ of the random variable $T$ in \eqref{h4} is negative.
    The constrained bootstrap test defined by \eqref{h5}  in Algortithm \ref{alg2} has asymptotic level $\alpha$ and is consistent. More precisely,
         }
       		\begin{itemize}
       			\item[a)] \textit{If  the null hypothesis $H_0: d_{1} \geq \epsilon$ holds, then $\lim_{n_{1}, n_{2} \to \infty} \mathbb{P}(\hat d_{1} < \hat{q}_{\alpha,1}^{*}) = 
       			0$ if $ \ d_{1} > \epsilon $
          and $\lim_{n_{1}, n_{2} \to \infty} \mathbb{P}(\hat d_{1} < \hat{q}_{\alpha,1}^{*}) = \alpha,$
          if $d_{1} = \epsilon$.
       			%$$ \lim_{n_{1}, n_{2} \to \infty} \mathbb{P}(\hat d_{1} < \hat{q}_{\alpha,1}^{*}) = \begin{cases}
       			%0, \ d_{1} > \epsilon \\
       			%\alpha, \ d_{1} = \epsilon.
       		%	\end{cases}$$
         }
       			\item[b)] \textit{If the alternative hypothesis $H_1: d_{1} < \epsilon$ holds, then  ~
       			$ 
       			\lim_{n_{1}, n_{2} \to \infty} \mathbb{P}(\hat d_{1} < \hat{q}_{\alpha,1}^{*}) = 1. 
       			$}
       		\end{itemize}
       		
       	\end{Thm}

       	\begin{rem}
       	\label{r1} ~~\\
        {\rm 
(a)         Let
  $
              \theta  =   \{ m_{1}(x,\beta_{1}) - m_{2}(x,\beta_{2}) \}_{x \in {\cal X}} $, 
              %$ \hat \theta    =   \{ m_{1}(x,\hat\beta_{1} ) - m_{2}(x,\hat\beta_{2}) \}_{x \in {\cal X}} $  
            $  \hat \theta = \{m_{1}(x,\hat{\hat{\beta_{1}}})-m_{2}(x,\hat{\hat{\beta_{2}}})\}_{x \in \mathcal{X}}$
              and $
              \hat \theta^*  =   \{ m_{1}(x,\hat\beta_{1}^{*}) - m_{2}(x,\hat\beta_{2}^{*}) \}_{x \in {\cal X}} 
       $   
            denote processes on $\ell ^{\infty} ({\cal X}) $
            and define the mapping $
        		\Phi(f) = \int_{\mathcal{X}} \abs{f(x)} \ts dx $ 
     from $\ell ^{\infty} ({\cal X}) $ onto $\mathbb{R} $. By the proof of Theorem
     \ref{thm1} we have
     $$
     \sup_{h \in BL}  \big | \mathbb{E} \big [ h  \big ( \sqrt{n} (  \Phi(\hat \theta  )- \Phi( \theta ))  \big ) \big ] -
     \mathbb{E} \big [  h \big (  \Phi^{'}_{ \theta }( \mathbb{G} ) \big ) ] \big |  = o(1) ~,
     $$
     where $BL$ denotes the space of bounded Lipschitz functions
     \citep[see][]{van2000asymptotic}, 
     $\mathbb{G}=\{G(x)\}_{x \in \mathcal{X}}$ is the Gaussian process defined in Theorem \ref{thm1}
     and  $\Phi^{'}_{ \theta }$ denotes the directional Hadamard derivative  at the process $\theta \in \ell ^{\infty} ({\cal X}) $. Now, according to Theorem 3.1  in \cite{fang2019inference},  the corresponding statement for the bootstrap process
    $$
     \sup_{h \in BL}  \big | \mathbb{E}^* \big [ h  \big ( \sqrt{n} (  \Phi(\hat \theta^*  )- \Phi( \hat \theta ))   \big ) ] -
     \mathbb{E} \big [  h \big (  \Phi^{'}_{ \theta }( \mathbb{G} ) \big ) ] \big |  = o_\mathbb{P} (1) 
     $$  
     (here  $\mathbb{E}^*$ denote the expectation conditional on the sample) holds if and only if the directional derivative of the mapping $\Phi$  at $\theta$ is linear, that is   $\Phi$ is Hadamard differentiable  at $\theta$. However, it follows from the proof of Theorem \ref{thm1} that this is only the case if 
            $\lambda(\mathcal{N}) = 0$.
            Thus in general,  it is not clear if the bootstrap is consistent in the case        $\lambda(\mathcal{N}) > 0$.
           \\
            However, from a practical point of view  the condition $\lambda(\mathcal{N}) = 0$ will be fulfilled in most applications. For example, if the predictor is one-dimensional the  curves  corresponding to typically used parametric regression models $m_1$ and $ m_2$ either intersect in  at most one point or they are completely identical (which is unlikely in most applications).
        Therefore, Theorem \ref{thm2} 
        is applicable in most applications and ensures that the parametric bootstrap defined by Algorithm \ref{alg2} yields a statistically valid procedure.  

            \smallskip

(b) A  bootstrap procedure, for which consistency can be proved  even in the case 
$\lambda ({\cal N)} >0$  can be obtained using the results of  \cite{fang2019inference},
in particular we will construct an estimator of the directional derivative of the $L_1$-norm mapping that fulfills the assumptions of Theorem 3.2 in their paper. 
 To be precise, we define
$$ \hat{\Phi}^{'}(f) := \int_{ | {\hat{\theta}}|  \geq {1}/{s_{n}} } \operatorname{sgn}(\hat{\theta}(x)) \ts f(x) \ dx \ + \int_{ |{\hat{\theta}|} < {1}/{s_{n}} } \abs{f(x)} \ dx $$
for some sequence $s_n$ satisfying $s_{n} / \sqrt{n} \to 0$. We then proceed as in Algorithm \ref{alg2}, where the steps (4) and (5) are replaced by   \\
(4') Calculate the bootstrap test statistic $ \hat{\Phi}^{'}(  \hat{\theta}^{*}-\hat{\theta} ) $
                   			    and  the 
                           $\alpha$-quantile $\hat{q}_{\alpha,1}^*$ of 
                          its distribution.
       	 \\
(5')	    	    The null hypothesis in \eqref{l3} is rejected, whenever
       %        	\begin{eqnarray}
        %       	\label{l44}
        $  \hat d_{1} < \hat{q}_{\alpha,1}^{*} + \epsilon. $
       	   %   	\end{eqnarray
               It can  then be shown 
               (see the Appendix) that the  statements of Theorem \ref{thm2}  hold for this test, even in the case  $\lambda ({\cal N}) >0 .$
}     
       	\end{rem}

        	\section{Finite sample properties}
        	\label{sec4}
  \def\theequation{4.\arabic{equation}}
  \setcounter{equation}{0}

We investigate the finite sample properties of the confidence intervals and the tests for the hypotheses \eqref{l3} by means of a small simulation study. For this purpose we consider two  E-max models 
   \begin{align}
       \label{hd22}
             m_{\ell }(x,\beta_{\ell}) &= \beta_{\ell 1} + \dfrac{\beta_{\ell 2}x}{\beta_{\ell 3}+x}
             ~~~(\ell=1,2), 
          \end{align}
          where 
          $x \in \mathcal{X} = [0,4]$.  We consider  two scenarios for the parameters:
          \begin{eqnarray}
          \label{inter}
          {\rm Intersecting~curves:} & \beta_{1} = (5,3,1)^\top~,~~ \beta_{2} = (5, 3+\gamma, 1+\gamma)^{\top},& \gamma \geq 0 \\
            {\rm Parallel  ~curves:} & \beta_{1} = (\delta,5,1)^{\top}  ~,~~    \beta_{2} = (0,5,1)^{\top},  & \delta \geq 0 .
            \label{par}
               \end{eqnarray}
  Some typical curves are displayed in Figure \ref{figplot}.

\begin{figure}[H]
    \centering%    \begin{minipage}{0.45\textwidth}
    \includegraphics[width=5cm, height=4 cm]{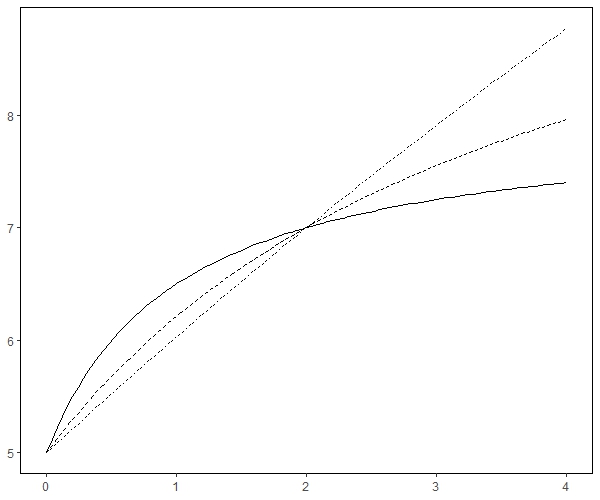}
    ~~~~   ~~~~
    \includegraphics[width=5cm, height=4cm]{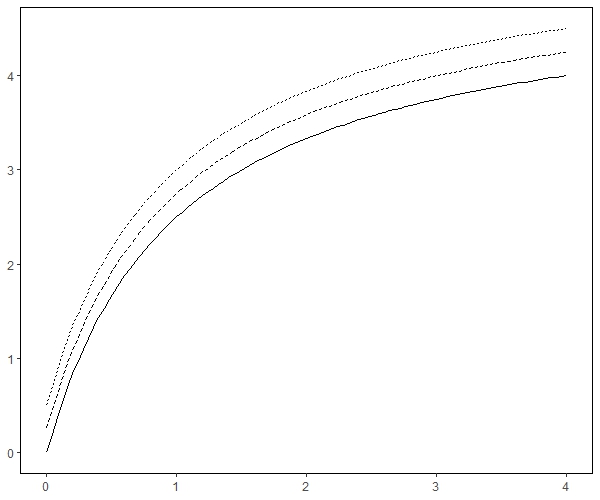}
    \caption{\it Typical E-max curves considered in the simulation study.
      Left panel: intersecting curves defined by \eqref{inter} with 
      $\gamma = 0$ (solid), $\gamma = 2.7 $  (dashed) and $\gamma = 30 $ (dotted). Right panel: parallel curves defined by \eqref{par} with 
      $\delta =0 $ (solid), $\delta = 1/4$ (dashed) and   $\delta =1 /2 $ (dotted).  }\label{figplot}
\end{figure}

          An equal number of observations is allocated to  five equidistant dose levels $x_{1,1} = x_{2,1} = 0, \ts x_{1,2} = x_{2,2} = 1, \ts  \ldots  , \ts 
          x_{1,5} = x_{2,5} = 4$, and   the sample sizes  for both groups are given by 
          $n_1=n_2 = 20, 50, 100 $ and $200$. 
The errors in the regression models 
  \eqref{l1} are centered normal distributions 	 
            with variances chosen as  $(\sigma_{1}^2,\sigma_{2}^2)=(0.25,0.25)$ and $(0.25,0.5)$.   All bootstrap results  are obtained by  $B=300$ replications.

    As estimators for the parameters in the regression models \eqref{l1}  we use least squares estimators, that is
\begin{eqnarray*}
       		%		   \label{h10}
       			\hat{\beta_{\ell}} &=&
       			{\rm arg}\min_{\beta_\ell  \in \mathbb{R}^{p_\ell} }
       			\dfrac{1}{n_{\ell}} \sum_{i=1}^{k_{\ell}}\sum_{j=1}^{n_{\ell,i}} (Y_{\ell,i,j}-m_{\ell}(x_{\ell,i}, {\beta_{\ell}}) )^2~ , \\
 %    \label{h12}
       			\hat\sigma_{\ell}^2 &= &\dfrac{1}{n_{\ell}} \sum_{i=1}^{k_{\ell}}\sum_{j=1}^{n_{\ell,i}} (Y_{\ell,i,j}-m_{\ell}(x_{\ell,i}, \hat{\beta_{\ell}}) )^2~
       			 \end{eqnarray*} 
$(\ell=1,2)$. We start investigating the coverage probabilities of the asymptotic and bootstrap confidence intervals for the distance $d_1$ defined in \eqref{conf} and \eqref{h7c}, respectively. For the asymptotic confidence interval we estimate the set 
${\cal N}$ by \eqref{h55} with $c=1$. 
The parameters in the two  E-max models \eqref{hd22}  are defined by 
\eqref{inter}  and \eqref{par} such that $d_1=1$.
The corresponding results are given in Table \ref{tab1}, where the upper part corresponds to the intersecting  and the lower part to the parallel scenario. We observe that the coverage probabilities of the asymptotic confidence interval are too small, but they improve with increasing sample size.  
The results for the bootstrap confidence intervals are more satisfactory.
For small sample sizes the bootstrap yields  intervals with a too large  coverage probability, but in general it provides an improvement.

\begin{table}[htbp]
    \centering
    \begin{tabular}{c| c|c||c|c|}
      & \multicolumn{4}{c}{$(\sigma_1^2, \sigma_2^2)$} \\
     \cline{2-5}
       $(n_{1}, n_{2})$  &  (0.25, 0.25) & (0.25, 0.5) & (0.25, 0.25) & (0.25, 0.5) \\
       \hline
            (20,20)      &  0.640        &  0.680     &  1.000   &  1.000             \\
            (50,50)      &  0.690        &  0.695     &  0.990   &  1.000             \\
            (100,100)    &  0.830        &  0.810     &  0.965   &  0.990             \\
            (200,200)    &  0.845        &  0.845     &  0.955   &  0.960             \\
    \hline
    \hline
%     & \multicolumn{4}{c}{$(\sigma_1^2, \sigma_2^2)$} \\
%     \cline{2-5}
%       $(n_{1}, n_{2})$  &  (0.25, 0.25) & (0.25, 0.5) & (0.25, 0.25) & (0.25, 0.5) \\
%       \hline
       %    (10,10)      &  0.60         & 0.66        &         1    &          1    \\
            (20,20)      &  0.620        &  0.645      &  1.000       &  1.000             \\
            (50,50)      &  0.695        &  0.760      &  0.994       &  1.000             \\
            (100,100)    &  0.775        &  0.755      &  0.972       &  0.984             \\
            (200,200)    &  0.855        &  0.840      &  0.938       &  0.960             \\
  %          (300,300)    &  0.940        &  0.905      &  0.956       &  0.956             \\
       %    (400,400)    & 0.92          & 0.94        & 0.948        &  0.932        \\
       %    (500,500)    & 0.95          & 0.95        & 0.946        &  0.952        \\
    \end{tabular}
    \caption{\it  
    Coverage probabilities of the asymptotic  (left)  and bootstrap $95\%$-confidence interval (right). The regression functions are given by \eqref{hd22} such
    that  $d_1 =1$. Upper part: intersecting curves defined by \eqref{inter}. Lower part: parallel curves defined by \eqref{par}. 
   }
    \label{tab1}
\end{table}

%\begin{table}[htbp]
%    \centering
 %   \begin{tabular}{c| c|c||c|c|}
 %    & \multicolumn{4}{c}{$(\sigma_1^2, \sigma_2^2)$} \\
%     \cline{2-5}
 %      $(n_{1}, n_{2})$  &  (0.25, 0.25) & (0.25, 0.5) & (0.25, 0.25) & (0.25, 0.5) \\
%       \hline
%            (20,20)      &  0.640        &  0.680     &  1.000   &  1.000             \\
%            (50,50)      &  0.690        &  0.695     &  0.990   &  1.000             \\
 %           (100,100)    &  0.830        &  0.810     &  0.965   &  0.990             \\
%            (200,200)    &  0.845        &  0.845     &  0.955   &  0.960             \\
%    \end{tabular}
 %   \caption{\it  
 %   Coverage probabilities of the asymptotic  (left)  and bootstrap $95\%$-confidence interval (right). The regression functions are given by \eqref{hd22} with 
%    $\beta_{1} = (5,3,1)^{\top}$
%          and $\beta_{2} = (5,3+\gamma,1+\gamma)^{\top}$
%    with $\gamma$ chosen such that  
%    $d_1 =1$.}
%    \label{tab1}
%\end{table}

 Next, we consider the problem of testing the hypotheses \eqref{l3} with  $\epsilon = 1$.
 We begin with the two tests 
 \eqref{h88} and \eqref{hd21}, which have been derived from the asymptotic and bootstrap confidence interval, respectively.
 The corresponding rejection probabilities are displayed in Figure \ref{fig1} for the E-max models \eqref{hd22} 
with  parameters  \eqref{par} (parallel curves). 
Note that in this case 
  $d_1= 4 \delta$ and 
  $    \lambda(\mathcal{N}) =0$,  whenever $\delta \not=0$.
Therefore, the cases $\delta \geq 0.25$  and $\delta <  0.25$ correspond to the null hypothesis and alternative in \eqref{l3}, respectively.
The horizontal solid line marks the significance level $\alpha = 0.05$ and the vertical solid  line 
corresponds to the boundary of the null hypothesis, that is $d_1 = 1 $.  The curves reflect the qualitative behaviour predicted by Theorem \ref{thm22} and Theorem \ref{thm23}. 
Note that the power curves are decreasing in $d_1$ as the null hypothesis  is given by  $H_0: d_1 \geq 1$. We observe that the asymptotic test  \textbf{\eqref{h88} }
does not keep its nominal level. In particular for small sample sizes the level is substantially exceeded. On the other hand, the bootstrap test 
\eqref{hd21}  keeps the nominal level in all  situations under consideration. The asymptotic test has more power, but this advantage 
comes at the cost of an unreliable approximation of the nominal level.
Therefore, if one would have to choose from these tests, we would recommend to use  the test based on the bootstrap confidence interval. Note that this test has not much power for the sample sizes $n_1=n_2=20$ and $50$, but the constrained bootstrap test  developed in Section \ref{sec3} will yield a further improvement.

\begin{figure}[H]
    \centering
    \includegraphics[width=5cm, height=4cm]{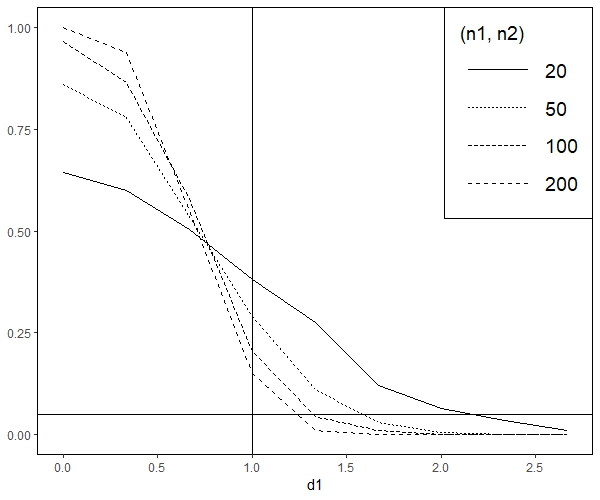}
   ~~~~  ~~~~
    \includegraphics[width=5cm, height=4cm]{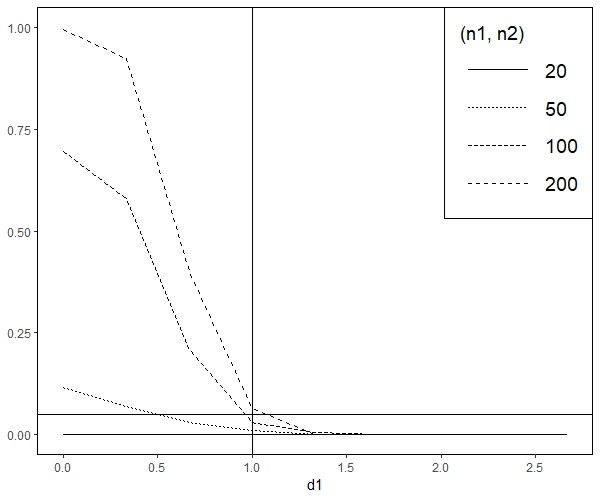}
    \includegraphics[width=5cm, height=4cm]{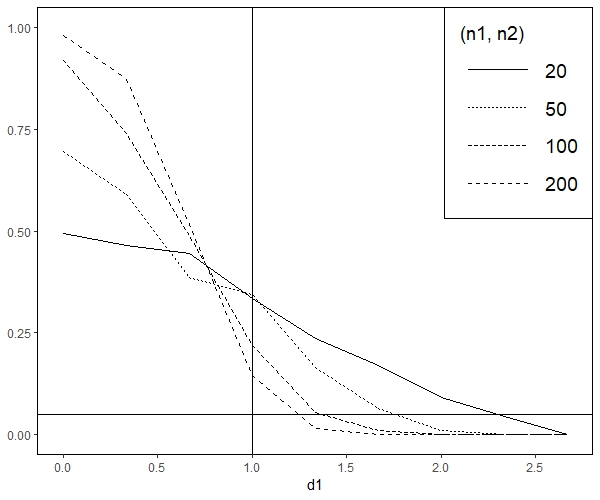}
   ~~~~   ~~~~
    \includegraphics[width=5cm, height=4cm]{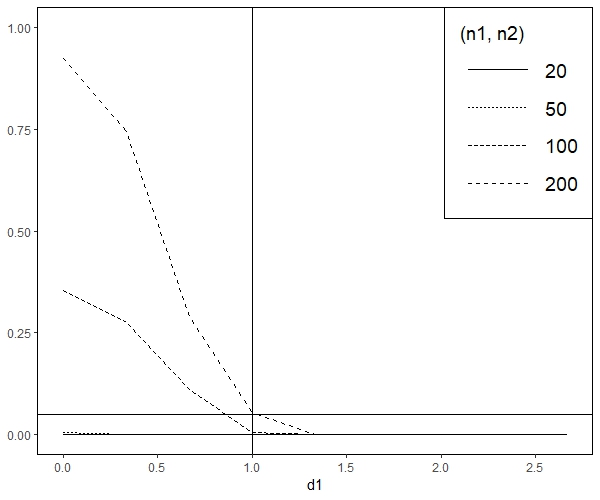}
    \caption{\it Rejection probabilities 
        of the tests \eqref{h88}   (left)  and 
              \eqref{hd21}  (right)  for the hypotheses \eqref{l3}.
            First row: $\sigma_1^2 =  \sigma_2^2 = 0.25$; Second row:   $\sigma_1^2 = 0.25,$ $ \sigma_2^2 = 0.5$.
            The nominal level is $\alpha =0.05$  and parallel E-Max models defined by \eqref{par} are considered.}
             \label{fig1}
\end{figure}

\begin{figure}[H]
    \centering
    \includegraphics[width=5cm, height=4cm]{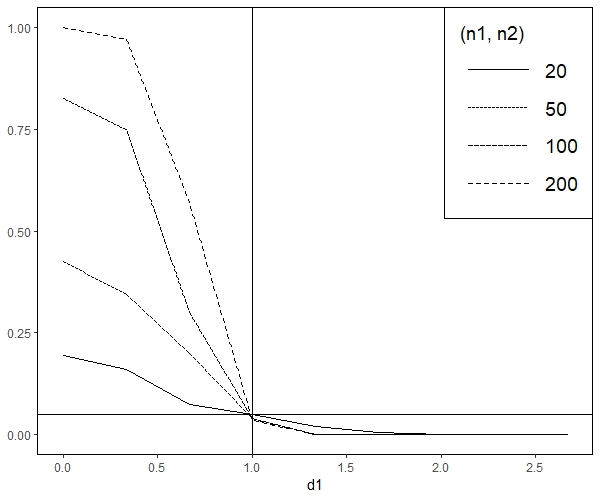}
~~~~ ~~~~
\includegraphics[width=5cm, height=4cm]{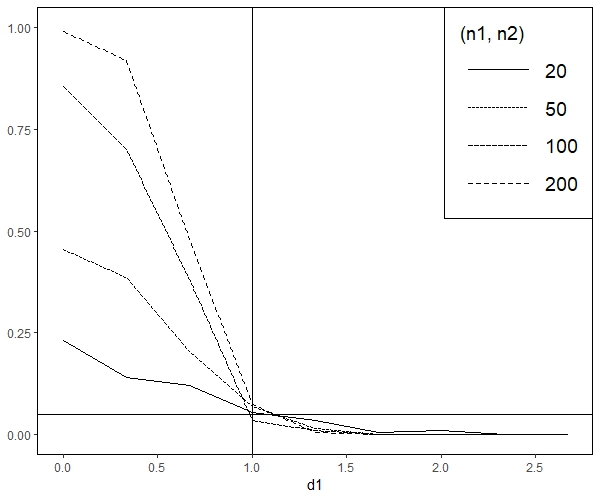}
    \includegraphics[width=5cm, height=4cm]{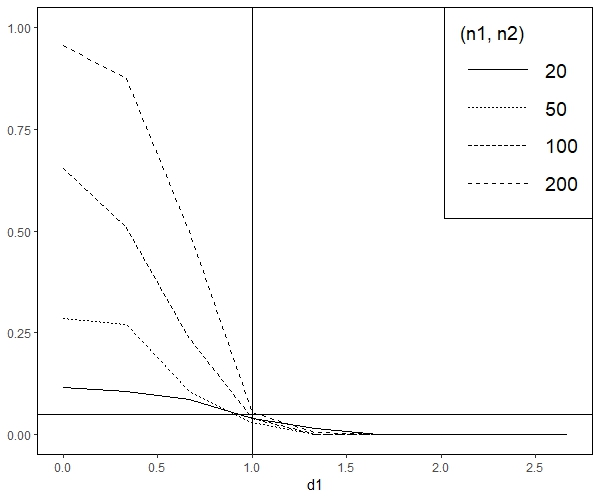}
~~~~ ~~~~
\includegraphics[width=5cm, height=4cm]{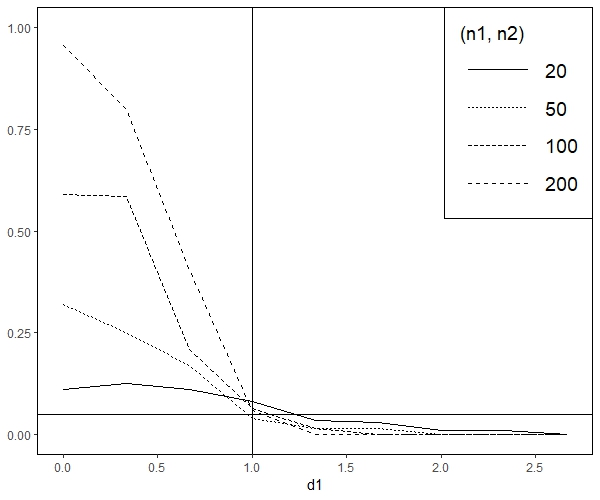}
\vspace{-0.25cm}
    \caption{\it Rejection probabilities 
        of the  constrained  bootstrap  test \eqref{h7} 
        (Algorithm \ref{alg2})
        for the hypotheses \eqref{l3}. Left panels: intersecting 
        E-max models with parameters \eqref{inter}
     Right panels: parallel 
        E-max models  with  parameters \eqref{par}. 
            First row: $\sigma_1^2 =  \sigma_2^2 = 0.25$; Second row:   $\sigma_1^2 = 0.25,$ $ \sigma_2^2 = 0.5$.
            The nominal level is $\alpha =0.05$.
            %The vertical line indicates the equivalence threshold $\epsilon$.
            }
             \label{fig1a}
\end{figure}

In Figure \ref{fig1a} we illustrate the performance of the test \eqref{h7}  (constrained  bootstrap  - see  Algorithm \ref{alg2}) for testing the hypotheses \eqref{l3} with $\epsilon =1$. The rejection probabilities for the situation investigated in 
Figure \ref{fig1} (parallel E-Max models defined by \eqref{par})
are shown in the right part of the figure.
          These results are directly comparable with the right panels in Figure \ref{fig1}.
          We observe that the constrained bootstrap  test \eqref{h7} yields a substantial improvement in  power compared to the test  \eqref{hd21}, which is based on the bootstrap confidence interval. For example, if 
$d_1=0$, $n_1=n_2= 50$, $\sigma_1^2=\sigma_2^2=(0.25,0.25)$,
    the test \eqref{hd21} has approximately power 
   $0.115$, 
while the power of the constrained bootstrap test is 
 $0.455$.   The left part of the   Figure  \ref{fig1a} shows the results for intersecting 
E-Max models (defined by \eqref{inter}).
A comparison   with the right part shows  that  the differences in power between the two cases (intersecting E-Max and shifted  E-Max curves) are rather small.

         \section{ Proofs}
        	\label{appendix}
           \def\theequation{5.\arabic{equation}}
  \setcounter{equation}{0}
        	
        	In this section we give proofs to all theoretical results in this work. For this purpose we require  the following assumptions:
        	\begin{description}
        		\item {\textbf{Assumption 1:} The errors $\eta_{\ell,i,j}$ have finite variance $\sigma_{\ell}^2$ and mean 0.}
        		
        		\item \textbf{Assumption 2:} The covariate region $\mathcal{X} \subset \mathbb{R}^{d}$ is compact and the number and location levels of $k_{\ell}$ does not depend on $n_{\ell}$ for $\ell =1,2$.
        		
        		\item {\textbf{Assumption 3:} All estimators of the parameters $\beta_{1}$,\ $\beta_{2}$ are computed over compact sets $B_{1} \subset \mathbb{R}^{p_{1}}$ and $B_{2} \subset \mathbb{R}^{p_{2}}$.}
        		
        		\item {\textbf{Assumption 4:} The regression functions $m_{1}$ and $m_{2}$ are twice continuously differentiable with respect to the parameters for all $b_{1}, b_{2}$ in the neighbourhoods of the true parameters $\beta_{1}, \beta_{2}$ and all $x \in \mathcal{X}$. The functions $(x,b_{\ell}) \rightarrowtail m_{\ell}(x,b_{\ell})$ and their first two derivatives are continuous on $ \mathcal{X} \times B_{\ell}$ for $\ell=1,2$.}
        		
        		\item {\textbf{Assumption 5:} Defining 
        		$$ \psi^{(n)}_{a,\ell}(b) \coloneqq \sum_{i=1}^{k_{\ell}} \dfrac{n_{\ell,i}}{n_{\ell}}(m_{\ell}(x_{\ell,i},a)-m_{\ell}(x_{\ell,i},b))^2,$$
        		we assume that for any $u >0$ there exists a constant $v_{u,\ell} > 0 $ such that 
        		$$  \liminf_{n \to \infty} \inf_{a \in B_{\ell}} \inf_{ \abs{b-a} \geq u} \psi^{(n)}_{a,\ell}(b) \geq v_{u,\ell}, \ \ \ell = 1,2.  $$}

        		\item {\textbf{Assumption 6:} The matrices $\Sigma_{\ell}$ are non-singular and the sample sizes $n_{1},n_{2}$ converge to infinity such that 
        		$$ \lim_{n_{\ell} \to \infty} \dfrac{n_{\ell,i}}{n_{\ell}} = \xi_{\ell,i} > 0, \ \ i = 1, \ldots ,k_{\ell}, \ \ell =1,2. $$ and
        		$$ \lim_{n_{1},n_{2} \to \infty} \dfrac{n}{n_{1}} = \kappa  \in (1,\infty). $$}
        		    		
        		\item \textbf{Assumption 7:} We denote by $\hat{\beta}_{1}, \hat{\beta}_{2}$ estimators of the parameters $\beta_{1}, \beta_{2}$ and assume that they can be linearized, meaning the estimators fulfill the following condition:
        		$$ \sqrt{n_{\ell}}(\hat{\beta_{\ell}}-\beta_{\ell}) = \dfrac{1}{\sqrt{n_{\ell}}}\sum_{i=1}^{k_{\ell}}\sum_{j=1}^{n_{\ell,i}} \phi_{\ell,i,j} + o_{\mathbb{P}}(1) \ \text{ as } n_{\ell} \to \infty, \ \ell =1,2$$
        		with square integrable influence functions $\phi_{1,i,j}$ and $\phi_{2,i,j}$ satisfying 
        		$$ \mathbb{E}[\phi_{\ell,i,j}] = 0, \ j=1, \ldots ,n_{\ell,i}, \ i = 1, \ldots ,k_{\ell}, \ \ell=1,2. $$
        		This implies that the asymptotic distribution of $\hat{\beta_{1}}$ and $\hat{\beta_{2}}$ is given by
        		$$ \sqrt{n_{\ell}} ( \hat{\beta_{\ell}} - \beta_{\ell} ) \xrightarrow{d} \mathcal{N}(0,\Sigma_{\ell}^{-1}), \ \ell =1,2,$$
        		where the asymptotic covariance matrix is given by
        		$$ \Sigma_{\ell}^{-1} = \sum_{i=1}^{k_{\ell}} \xi_{\ell,i} \mathbb{E}[\phi_{\ell,i,j}\phi_{\ell,i,j}^{\top}], \ \ell =1,2.$$
          Moreover, the variance estimators $\hat\sigma_{1}^2$ and $\hat\sigma_{2}^2$  used  in \eqref{h6} are consistent.
    	
        	\end{description}
        	
        	\subsection{Proof of Theorem \ref{thm1}}
        		We will prove this result by an application of the (functional) delta method for {\it directionally} differentiable functionals as stated in Theorem 2.1 in \cite{shapiro1991asymptotic}. 
          We introduce the notations
           $ \theta (x)  =  m_{1}(x,\beta_{1}) - m_{2}(x,\beta_{2}) $,
           $ \hat \theta (x)  =  m_{1}(x,\hat\beta_{1}) - m_{2}(x, \hat \beta_{2}), $ 
           $ \theta  =   \{ \theta (x)   \}_{x \in {\cal X}} $ and $ \hat \theta  =   \{\hat\theta (x)   \}_{x \in {\cal X}} $, and 
          will show  below that the mapping
        			\begin{eqnarray}
        				\label{l40}
        				\Phi: 
            \begin{cases}
        &  \ell^{\infty}(\mathcal{X}) \to \mathbb{R} \nonumber \\ 
        	     &    f \to 	\Phi(f) = \int_{\mathcal{X}} \abs{f(x)} \ts dx
                     \end{cases} 
        				\end{eqnarray}
            is {\it directionally } Hadamard differentiable with respect to the $L^{1}$-norm on $\ell^{\infty}(\mathcal{X})$ and the absolute value norm on $\mathbb{R}$, where the
            derivative is given by   
        			\begin{equation*}
        			        	%\label{l41}
       \Phi_{h}^{'}:~     \begin{cases}
      				&  \ell^{\infty}(\mathcal{X}) \to \mathbb{R} 
            \\
        		& f \to 		\Phi_{h}^{'}(f) = \int_{ \{ h \neq 0 \} } \operatorname{sgn}(h(x))f(x) \ts dx \ + \int_{ \{ h = 0 \} } \abs{f(x)} \ts dx 
    \end{cases} 
       			\end{equation*}           
        				at $ h \in \ell^{\infty}(\mathcal{X})$. Note that $\left(\ell^{\infty}(\mathcal{X}),\norm{\cdot}_1\right)$  is still separable and that its norm is weaker than the sup-norm. 

            Hence the convergence in distribution
        				\begin{align*}
        				     \sqrt{n}\big \{ \hat \theta (x)-  \theta(x)\big \}_{x \in \mathcal{ X}} \xrightarrow{d} \{G(x)\}_{x \in \mathcal{X}}
        				\end{align*}
             in $\left(\ell^{\infty}(\mathcal{X}),\norm{\cdot}_\infty\right)$ 
        				established in \cite{dette2018equivalence} is also valid in this setting. In particular, applying the (directional) delta method  (Theorem 2.1 in \cite{shapiro1991asymptotic}) gives 
        		\begin{align*}
        		\sqrt{n}(\hat d_{1}-d_{1}) &= \sqrt{n}\int_{\mathcal{X}} \ | {\hat \theta(x)}|  - \abs{ \theta(x)} \ dx =  \sqrt{n} \left( \Phi( \{\hat \theta(x) \}_{x \in \mathcal{X}} ) - \Phi( \{ \theta(x) \}_{x \in \mathcal{X}}) ) \right) \\
        		&\xrightarrow{d} \Phi^{'}_{ \theta  }( \{G(x)\}_{x \in \mathcal{X}}  ) = \int_{{{\cal N}}^c} \operatorname{sgn}(\theta(x)) \ G(x) \ dx 
          + \int_{{{\cal N}}} \abs{G(x)} \ dx,
        		\end{align*}
          where ${{\cal N}}$ is defined in \eqref{h1}.
Therefore, we   are left with showing the differentiability of the functional $\Phi$. For this purpose we  write    $ \Phi = \Phi_{1} \circ \Phi_{2}, $
        		where
        		\begin{align*}
        		\label{b1}
        		\Phi_{1}: &
          \begin{cases}
        &  \ell^{\infty}(\mathcal{X}) \to \mathbb{R} \\
        	& f \to 	\Phi_{1}(f) = \int_{\mathcal{X}} f(x) dx ~
\end{cases}~,  ~~~~
        		\Phi_{2}:     
        \begin{cases}
        &  \ell^{\infty}(\mathcal{X}) \to \l^{\infty}(\mathcal{X}) \\
        	&	\Phi_{2}(f) = \abs{f}
          \end{cases} 
        		\end{align*}
As a linear mapping $\Phi_{1}$ is obviously Hadamard differentiable with derivative $(\Phi_{1})_{h}^{'} = \Phi_{1}$ at $ h \in \ell^{\infty}(\mathcal{X})$, where the latter is equipped with the $L_1$ norm. 
        		We prove below that
 $\Phi_{2}$ is directionally Hadamard differentiable with respect to the $L^{1}$-norm on $\ell^{\infty}(\mathcal{X})$ and derivative 
        			\begin{eqnarray}
           \label{hd1}
        			(\Phi_{2})_{h}^{'}:
           \begin{cases}
          &  \ell^{\infty}(\mathcal{X}) \to \ell^{\infty}(\mathcal{X}) \\
        			&  f \to 
           (\Phi_{2})_{h}^{'}(f) = \mathbbm{1}_{\{ h \neq 0 \}} \operatorname{sgn}(h)f + \mathbbm{1}_{\{h = 0\}} \abs{f}   
           \end{cases}
        			\end{eqnarray}
        			at $ h \in \ell^{\infty}(\mathcal{X})$.
 The assertion then follows by the chain rule given in Proposition 3.6 in \cite{shapiro1990concepts}.
        
        	For a proof of \eqref{hd1}
let $(x_{n})$ be a sequence in $\l^{\infty}(\mathcal{X})$ converging to $x$ and $(t_{n})$ be a sequence of positive real numbers converging to zero. We show that 
        	\begin{equation} 
      \label{hd2}   
           \norm{
          \frac{\Phi_{2}(h+t_{n}x_{n}) - \Phi_{2}(h)}{t_{n}} - (\Phi_{2})_{h}^{'}(x)
          }_{1} \xrightarrow{n \to \infty} 0 ,
        \end{equation}
          where 
          $$
          Z_{n} = \frac{\Phi_{2}(h+t_{n}x_{n}) - \Phi_{2}(h)}{t_{n}} - (\Phi_{2})_{h}^{'}(x) ~.
          $$
    This proves the claim. 
        		
   For a proof of \eqref{hd2} note that this statement is equivalent to 
        		\begin{align*}
        		%Z_{n} \xrightarrow{\norm{\cdot}_{1}} 0 \iff &
          \text{(A2.1)}  & ~~~~~~ \ Z_{n}  \xrightarrow{\lambda} 0  ~, \\
        		\text{(A2.2)}  & ~~~~~~  \ (Z_{n}) \text{ is uniformly integrable }, 
        		\end{align*}
    where $\xrightarrow{\lambda}$ denotes $\lambda$-stochastic convergence  \citep[see Theorem 21.4 and the preceding definitions in][]{bauer2011measure}.
        		
        		\textit{Proof (A2.1).} To prove this statement, it suffices to show that every subsequence $(Z_{n_{k}})$ of $(Z_{n})$ has a further subsequence $(Z_{n_{k_{j}}})$ which converges to zero almost everywhere. So let $(Z_{n_{k}})$ be a subsequence of $(Z_{n})$. Since $x_{n} \xrightarrow{\norm{\cdot}_{1}} x$ by assumption, we know that $x_{n_{k}} \xrightarrow{\norm{\cdot}_{1}} x$. Theorem 15.7 in \cite{bauer2011measure} then implies that there exists a subsequence $(x_{n_{k_{j}}})$ such that $x_{n_{k_{j}}} \xrightarrow{a.e} x$. We conclude that $Z_{n_{k_{j}}} \xrightarrow{a.e} 0$ by the following:
        		\begin{itemize}
        			\item[1)] On the set $\{ t \in \mathbb{R} \mid x_{n_{k_{j}}}(t) \to x(t), \ h(t) = 0 \}$ we have 
           $$
           Z_{n_{k_{j}}}(t) = \big | {x_{n_{k_{j}}}(t)} \big | -\abs{x(t)} \xrightarrow{ j \to \infty} 0.
           $$
        			\item[2)] On the set $\{ t \in \mathbb{R} \mid x_{n_{k_{j}}}(t) \to x(t), \ h(t) > 0 \}$ we have for sufficiently large $j$ $$
           Z_{n_{k_{j}}}(t) = x_{n_{k_{j}}}(t)-x(t) \xrightarrow{ j \to \infty} 0.
           $$
        			 \item[3)] On the set $\{ t \in \mathbb{R} \mid x_{n_{k_{j}}}(t) \to x(t), \ h(t) < 0 \}$ we have for sufficiently large $j$ $$
            Z_{n_{k_{j}}}(t) = - x_{n_{k_{j}}}(t)+x(t) \xrightarrow{ j \to \infty} 0.
            $$
         		\end{itemize}
         		\textit{Proof (A2.2).} We have that 
         		\begin{align*}
         		 \abs{Z_{n}} &= \abs{\frac{\abs{h+t_{n}x_{n}} - \abs{h}}{t_{n}} - (\Phi_{2})_{h}^{'}(x)}  
         		 \leq \abs{\frac{\abs{h+t_{n}x_{n}} - \abs{h}}{t_{n}}} + \abs{(\Phi_{2})_{h}^{'}(x)} \\
         		 &\leq \frac{ \abs{h+t_{n}x_{n} - h}  }{t_{n}} + \abs{x} 
         		 = \abs{x_{n}} + \abs{x}
         		\end{align*}
         		where we have used 
           the 
           % triangle inequality for the first inequality and the reverse triangle inequality and 
           definition of $(\Phi_{2})_{h}^{'}$ for the second inequality. Now $\abs{x_{n}}$ is uniformly integrable, since $x_{n} \xrightarrow{\norm{\cdot}_{1}} x$ and  $\abs{x}$ is uniformly integrable, since $x$ is bounded. Therefore $\abs{x_{n}} + \abs{x}$ is uniformly integrable. Since $\abs{Z_{n}}$ is dominated by $\abs{x_{n}} + \abs{x}$, the result follows.

        	\subsection{Proof of Theorem \ref{thm22}}

          We first show that, unconditionally, 
          \begin{eqnarray}
              \label{det100}
           \hat T \xrightarrow{\mathbb{P}} T~.   
          \end{eqnarray}
          For this purpose 
          we note  that it follows from 
          Assumption 1-7 that 
                \begin{eqnarray}
         \label{p1}
         \| {\theta-\hat \theta} \|_\infty=O_\mathbbm{P}(n^{-1/2})    .
         \end{eqnarray}
         Next we   define 
     \begin{align*}
     S:=\int_{
        		{\cal N}^c} \operatorname{sgn}\theta(x) \ \hat G(x) \ dx  + \int_{
        		{\cal N} }  | \hat G(x) | \ dx      
     \end{align*}
     and let   $\lambda $ denote the Lebesgue measure on $\mathcal{X}$. By the triangle inequality we have
     \begin{eqnarray*}
         |{T - \hat T} | \leq \abs{T-S}+ | {S-\hat T} |.
     \end{eqnarray*} 
     Note that $ \left \{\hat G(x)-G(x) \right\}_{x \in \mathcal{X}}$ tends to  $0 \in (\ell_\infty(\mathcal{X}),\norm{\cdot}_\infty)$ in probability by Slutsky's Theorem, 
     which implies 
     $$
  |  S-T |  \leq \| {\hat G-G} \| _1\leq \lambda (\mathcal{X})\|{\hat G-G}\| _\infty~,
  $$
  and as a consequence $\abs{T-S} \to 0
     $
     in probability. We are hence left with showing that $| {S-\hat T} | \to 0$ in probability. To this  end we observe
        		\begin{eqnarray*}
        		    \vert S - \hat T\vert\leq  A+B
        		\end{eqnarray*}
        		where
        		\begin{align*}
        		    A&=\Big \vert \int_{ \hat{\mathcal{N}}^c } \operatorname{sgn}(\hat \theta(x))\ts \hat{G}(x) \ dx \ - \int_{ \mathcal{N}^c }\operatorname{sgn}(\theta(x))\ts \hat{G}(x)\ dx  \Big \vert\\
        		    B&=\Big \vert \int_{ \hat{\mathcal{N}} } \vert\hat{G}(x)\vert \ dx \ - \int_{ \mathcal{N}}\vert \hat{G}(x) \vert \ dx  \Big \vert
        		\end{align*}
        		We will only show that $A\to 0$ in probability, the 
        		corresponding result for $B$ follows by similar arguments. Note that
        		\begin{align*}
        		 A &\leq \Big \vert \int_{ \hat{\mathcal{N}}^c \cap \mathcal{N}  } \operatorname{sgn}(\hat \theta(x))\ts \hat{G}(x) \ dx 
      		 - 
        		 \int_{\mathcal{N}^c \cap \hat{\mathcal{N}} }\operatorname{sgn}(\theta(x))\ts \hat{G}(x)\ dx  \Big \vert\\& \quad \quad + \Big \vert\int_{\mathcal{N}^c\cap \hat{\mathcal{N}}^c } \big ( \operatorname{sgn}(\hat \theta(x)) -\operatorname{sgn}(\theta(x))
        		 \big ) \hat G(x) dx
        		 \Big \vert\\
        		 &\leq \lambda (\hat{\mathcal{N}}^c \cap \mathcal{N})\| {\hat G} \| _\infty+\lambda (\mathcal{N}^c \cap \hat{\mathcal{N}})\| {\hat G} \| _\infty+o_\mathbbm{P}(1)~, 
        		\end{align*}
        		where the last inequality is true because with high probability  the signs in the third integral cancel each other other out on $\hat{\mathcal{N}}^c \cap \mathcal{N}^c$.
        		This can be seen by  recalling the definition of the set  $\mathcal{N}^c$ and equation \eqref{p1}. The other two terms vanish due to $\hat G(x)$ being bounded in probability by virtue of its tightness and because
        		\begin{align*}
        		    \lambda (\hat{\mathcal{N}}^c \cap \mathcal{N})&=\lambda 
        		    \big ( \{ x ~|~\hat \theta(x) \geq c\sqrt{\log(n)/n} , \theta(x)=0 \} \big ) \\
        		    &\leq \lambda  \big ( \{ x ~|~ \sqrt{n}(\hat \theta(x)-\theta(x))\geq c\log(n) \} \big )  =o_\mathbbm{P}(1)
        		\end{align*}
        		due to \eqref{p1}. As a similar inequality holds true for the set $\mathcal{N}^c \cap \hat{\mathcal{N}}$, this concludes the proof of \eqref{det100}. \\

           % \lk{Hier müssen wir eventuell Notation ändern? Die Daten $Y_1, .. , Y_n$ haben wir anders indiziert. Eventuell könnte man $\mathcal{Y} := \{   Y_{\ell,i,j}  : \ell = 1,2, i = 1, ... , k_{\ell}, j=1,\ldots , n_{\ell, i} \} $ für die Daten schreiben. Das würde gut mit dem Beweis von Theorem 3.3 passen, da verwenden wir dieselbe Notation für die Daten. Jedenfalls geht aus bisherigem Aufschrieb nicht hervor was $Y_1, .. , Y_n$ sein sollen.} 
            
            %Define $\mathcal{Y} := \{   Y_{\ell,i,j}  : \ell = 1,2, i = 1, ... , k_{\ell}, j=1,\ldots , n_{\ell, i} \} $ and consider the distribution $\mathbb{P}^{\hat{T}\vert Y_1,...,Y_n}$ and note that by the previous argument we have

            Define $\mathcal{Y} := \{   Y_{\ell,i,j}  : \ell = 1,2, \ i = 1, ... , k_{\ell}, \ j=1,\ldots , n_{\ell, i} \} $, consider the conditional distribution $\mathbb{P}^{\hat{T}\vert  \mathcal{Y} }$ and note that by the previous argument we have
            \begin{align*}
                \mathbb{P}(\hat{T}-T \in \mathcal{A} )=\int \mathbb{P}^{\hat{T}-T\vert \mathcal{Y}  }(\mathcal{A})d\mathbb{P} \rightarrow 0
            \end{align*}
            which implies that $\hat{T}-T\vert \mathcal{Y} \rightarrow 0$ in probability by a suitable choice of a countable family of $\mathcal{A}$ and a repeated subsequence argument.  We hence obtain that $\hat q_{0, \alpha}$ converges to $q_\alpha$ in probability. As all quantities of which we take limits in the following are real valued we may assume WLOG that it does so even almost surely.\\
            Observe that
         \begin{align*}
          \mathbb{P}\Big( d_1 \in [0, \hat{d}_1 - \dfrac{\hat q_{0, \alpha}}{\sqrt{n}}] \Big) &= 1 - \mathbb{P}\Big(\sqrt{n}(\hat{d}_1 - d_1 )  \leq \hat q_{0, \alpha} \Big)\\
          &= 1 - \mathbb{P}\Big(\sqrt{n}(\hat{d}_1 - d_1 )  \leq  q_\alpha+o(1) \Big)
        \end{align*}
        By Egorovs Theorem we may assume that the o(1) term vanishes uniformly on a set of measure $\delta$ for any $\delta >0$, to be precise for $n(m)$ large enough we have $o(1)\leq 1/m$ on a set $\mathcal{A}_m$ that has measure at least $1-1/m$. Hence for $n\geq n(m)$ we obtain
        \begin{align*}
            \mathbb{P}\Big(\sqrt{n}(\hat{d}_1 - d_1 )  \leq  q_\alpha+o(1) \Big)&\leq  
            \mathbb{P}\Big(\sqrt{n}(\hat{d}_1 - d_1 )  \leq  q_\alpha+1/m,\mathcal{A}_m\Big)+\mathbb{P}(\mathcal{A}_m^c)\\ &\leq \mathbb{P}\Big(\sqrt{n}(\hat{d}_1 - d_1 ) \leq  q_\alpha+1/m\Big)+1/m
        \end{align*}
        A similar lower bound can be obtained by the same arguments. Letting $n$ go to infinity then establishes 
        \begin{align*}
            \mathbb{P}\Big( d_1 \in [0, \hat{d}_1 - \dfrac{\hat q_{0, \alpha}}{\sqrt{n}}] \Big) = 1-\mathbb{P}\Big(\sqrt{n}(\hat{d}_1 - d_1 )  \leq  \hat q_{0, \alpha} \Big)\rightarrow 1-\alpha
        \end{align*}
       because the convergence of the distribution functions of $\sqrt{n}(\hat d_{1} - d_{1}) $ is uniform for all continuity points of  $F_T$.
    
      This proves the first part of Theorem \ref{thm22}.\\

        For the test in \eqref{h88} we have under the null hypothesis 
    $H_0: d_1 \ge \epsilon $ that
    $\epsilon - d_{1} \leq 0$,
    which implies for the probability of rejection
        	\begin{align*}
        		\mathbb{P}\Big (\hat d_{1} < \epsilon + \dfrac{\hat q_{0, \alpha}}{\sqrt{n}} \Big ) &= \mathbb{P}(\sqrt{n}(\hat d_{1} - d_{1}) < \sqrt{n}(\epsilon - d_{1}) + \hat q_{0, \alpha}) \\
        		&\leq \mathbb{P}(\sqrt{n}(\hat d_{1} - d_{1}) < \hat q_{0, \alpha}) \\
        		&= \mathbb{P}(\sqrt{n}(\hat d_{1} - d_{1}) < q_{\alpha} - (q_{\alpha} - \hat q_{0, \alpha}) )
        		\xrightarrow{n \to \infty} \alpha, \label{l46}
        		\end{align*}		
        	 where the convergence follows from similar arguments as for the first part of the theorem and Theorem \ref{thm1}. Consequently,  the  decision rule  \eqref{h88} defines an asymptotic level $\alpha$-test. Similarly, under the alternative, we have
                $\epsilon - d_{1} > 0$, which yields consistency, that is
        		\begin{align*}
        		\mathbb{P} \Big (\hat d_{1} < \epsilon + \dfrac{\hat q_{0, \alpha}}{\sqrt{n}} \Big ) &= \mathbb{P}(\sqrt{n}(\hat d_{1} - d_{1}) < \sqrt{n}(\epsilon - d_{1}) + \hat q_{0, \alpha}) \xrightarrow{n \to \infty} 1,
        		\end{align*} 
    since $\hat q_{0, \alpha} \xrightarrow{\mathbb{P}} q_{\alpha}$ and $\sqrt{n}(\epsilon - d_{1}) \xrightarrow{n \to \infty} \infty$ imply  $\sqrt{n}(\epsilon - d_{1}) + \hat q_{0, \alpha} \xrightarrow{n \to \infty} \infty$ and we know that $\sqrt{n}(\hat d_{1} - d_{1})$ converges in distribution by Theorem \ref{thm1}. 

\subsection{Proof of Theorem \ref{thm23}}
We start with proving the properties of the test. We have
\begin{align*}
  \mathbb{P}\big(  \hat{q}_{1-\alpha,0 }^{*}< \epsilon \big)  &= \mathbb{P}\big(\sqrt{n}(\hat{q}_{1-\alpha,0 }^{*}-\hat d_1)< \sqrt{n}(\epsilon-d_1)+\sqrt{n}(d_1-\hat d_1)\big) 
\end{align*}

Following the arguments in the proof of Theorem 2 of \cite{dette2018equivalence} (where we use Theorem 23.9 from \cite{van2000asymptotic} instead of an explicit first order expansion and the continuous mapping theorem) we obtain that 
\begin{align}
    \sqrt{n}(\hat{q}_{1-\alpha,0 }^{*}-\hat d_1) \xrightarrow{\mathbb{P}} q_{1-\alpha} 
\end{align}
where  $q_{1-\alpha}$ is the $1-\alpha$ quantile of the Random Variable $T$ defined in \eqref{h4}. Since $\sqrt{n}(d_1-\hat d_1)$ converges in distribution to $T$ by Theorem \ref{thm1}, $T$ is symmetric when $\lambda(\mathcal{N})=0$, $\sqrt{n}(\epsilon-d_1)$ converges to zero if $d_1=\epsilon$ and to $\pm \infty$ in the alternative/remainder of the null hypothesis, we obtain the desired statement on the significance level and the consistency of the test.\\

For the confidence interval we observe that
\begin{align*}
    \mathbb{P}\big( d_1 \in \hat I_n^{*}\big)&= \mathbb{P}\big(  d_1< \hat{q}_{1-\alpha,0 }^{*} \big) = \mathbb{P}\big(\sqrt{n}(d_1-\hat d_1)< \sqrt{n}(\hat{q}_{1-\alpha,0 }^{*}-\hat d_1)\big) 
\end{align*}
which yields the desired statement by \eqref{l46}.
        	\subsection{Proof of Theorem \ref{thm2}}
         
      {\bf  Proof of a).} 
       	 	First, we determine the asymptotic distribution of the bootstrap test statistic $\hat d_{1}^{*}$. Define $\hat \theta^{*}(x) = m_1(x, \hat \beta_1^*) - m_2(x, \hat \beta_2^*) $ and $\hat{\hat{\theta}} (x) = m_1(x, \hat{\hat{\beta}}_1) - m_2(x, \hat{\hat{\beta}}_2)$.  Following the proof of Theorem 1 in \cite{dette2018equivalence} yields that conditionally on $\mathcal{Y}$ in probability
       			\begin{eqnarray*}
       			   % \label{l11}
       			   \big \{\sqrt{n}\big (\hat \theta^*(x) -\hat{\hat{\theta}}(x) \big ) \big \}_{x \in \mathcal{X}} \xrightarrow{d} \{G(x)\}_{x \in \mathcal{X}}.
       			\end{eqnarray*}
          By assumption the directional hadamard derivative $\Phi'_{\theta}$ is linear and thus a proper hadamard derivative which allows us to apply the delta method for the bootstrap as stated in Theorem 23.9 in \cite{van2000asymptotic}. Consequently we obtain
       		\begin{eqnarray}
       		\sqrt{n} \big (\hat d_{1}^{*} - \hat{\hat d}_{1} \big )
       		&= \sqrt{n} \big ( \Phi( \{ \hat \theta^*(x) \}_{x \in \mathcal{X}} ) -  \Phi( \{ \hat{\hat{\theta}}(x) \}_{x \in \mathcal{X}} ) \big ) \nonumber  \xrightarrow{d} \Phi^{'}_{ \theta  }( \{G(x)\}_{x \in \mathcal{X}} ) 
       		\end{eqnarray}
       		conditionally on $\mathcal{Y}$ in probability. \\

       		\textit{\underbar{Case 1:  $d_{1} > \epsilon$.}} We observe that
       		\begin{eqnarray}
         \nonumber 
       		\mathbb{P}(\hat d_{1} < \hat{q}_{\alpha,1}^* ) &=& \mathbb{P}(\hat d_{1} < \hat{q}_{\alpha,1}^* \ , \hat d_{1} \geq \epsilon ) + \mathbb{P}(\hat d_{1} < \hat{q}_{\alpha,1}^* \ , \hat d_{1} < \epsilon ) \  \\
               \nonumber 
       		&\leq& \mathbb{P}(\hat d_{1} < \hat{q}_{\alpha,1}^* \ , \hat{\hat d}_{1} = \hat d_{1} ) + \mathbb{P}( \hat d_{1} < \epsilon)  
         \nonumber \\
       		&\leq& \mathbb{P}( \hat{\hat d}_{1} < \hat{q}_{\alpha,1}^* ) + \mathbb{P}( \sqrt{n}( \hat d_{1} - d_{1} ) < \sqrt{n}(\epsilon-d_{1}) ). \label{l15}
       		\end{eqnarray}
       		We now show that the first sequence in the upper bound \eqref{l15} converges to zero. 
       		To prove this, first note that for all $\alpha \in (0,1) $
       		\begin{eqnarray}
       		\label{l16}
       		\sqrt{n} \ (\hat{q}_{\alpha,1}^*-\hat{\hat d}_{1}) \xrightarrow{\mathbb{P}} q_{\alpha}, 
       		\end{eqnarray}
       		where $q_{\alpha}$ denotes the $\alpha$-quantile of the random variable $T$ defined in \eqref{h4}. To see this, observe that 
       		\begin{eqnarray*}
       		\alpha &=& \mathbb{P}( \hat d_{1}^{*} < \hat{q}_{\alpha,1}^* \mid \mathcal{Y}  ) =  \mathbb{P}( \sqrt{n} \ts (\hat d_{1}^{*}-\hat{\hat d}_{1}) < \sqrt{n} ( \hat{q}_{\alpha,1}^*-\hat{\hat d}_{1} )  \mid \mathcal{Y} ) \ \text{ a.s }. 
       		\end{eqnarray*}
       		Since $\sqrt{n}( \hat d_{1}^{*} - \hat{\hat d}_{1} )$ converges in distribution to $T$ conditionally on $\mathcal{Y}$ in probability, Lemma 21.2 in \cite{van2000asymptotic} yields the result \eqref{l16}. Using \eqref{l16} and choosing $\alpha > 0$ small enough such that $q_\alpha < 0$, we obtain
       		\begin{eqnarray*}
       		\mathbb{P}( \hat{\hat d}_{1} < \hat{q}_{\alpha,1}^* ) &=& \mathbb{P}( \sqrt{n}( \hat{q}_{\alpha,1}^* - \hat{\hat d}_{1}) > 0) 
       		\leq \mathbb{P} \big (  \big | { \sqrt{n}( \hat{q}_{\alpha,1}^* - \hat{\hat d}_{1}) - q_{\alpha} } \big| > - q_{\alpha} \big   ) \xrightarrow{n \to \infty} 0.
       		\end{eqnarray*}
       		Finally, we show that the second sequence in the upper bound \eqref{l15} converges to zero. Since $d_{1} > \epsilon$ by assumption, we have that $\sqrt{n}(\epsilon - d_{1}) \to -\infty$ and from Theorem \ref{thm1} we know that $\sqrt{n}(\hat d_{1} - d_{1})$ converges in distribution. Therefore, the result follows. This concludes the proof of a) in the case $d_{1} > \epsilon$. \\
       		
       		\textit{\underbar{Case 2: $d_{1} = \epsilon$}.}
       		We observe that
       		\begin{eqnarray*}		
       		\mathbb{P}( \hat d_{1} < \hat{q}_{\alpha,1}^* ) &=& \mathbb{P}( \hat d_{1} < \hat{q}_{\alpha,1}^* \ , \hat d_{1} \geq \epsilon) + \mathbb{P}( \hat d_{1} < \hat{q}_{\alpha,1}^* \ , \hat d_{1} < \epsilon ) \label{l20}  \\
       		&=& \mathbb{P}( \hat d_{1} < \hat{q}_{\alpha,1}^* \ , \hat{\hat d}_{1} = \hat d_{1} ) + \mathbb{P}( \hat d_{1} < \hat{q}_{\alpha,1}^* \ , \hat{\hat d}_{1} = \epsilon ) - \mathbb{P}( \hat d_{1} < \hat{q}_{\alpha,1}^* \ , \hat d_{1} = \epsilon ) \label{l21} \\
       		&=& \mathbb{P}( \hat d_{1} < \hat{q}_{\alpha,1}^* \ , \hat{\hat d}_{1} = \hat d_{1} ) + \mathbb{P}( \hat d_{1} < \hat{q}_{\alpha,1}^* \ , \hat{\hat d}_{1} = \epsilon = d_{1} ) + o(1) \label{l22} \\
       		&=& \mathbb{P}( \sqrt{n}(\hat d_{1} - d_{1}) < \sqrt{n}(\hat{q}_{\alpha,1}^* - \hat{\hat d}_{1}) \ , \hat{\hat d}_{1} = \epsilon ) + o(1) \label{l23} \\
       		&=& \mathbb{P}( \sqrt{n}(\hat d_{1} - d_{1}) < \sqrt{n}(\hat{q}_{\alpha,1}^* - \hat{\hat d}_{1})) \label{l24} \\
                && - \mathbb{P}( \hat d_{1} - d_{1} < \hat{q}_{\alpha,1}^*  - \hat{\hat d}_{1} \ , \hat{\hat d}_{1} > \epsilon) + o(1). 
       		\end{eqnarray*}
       		Because of \eqref{l16} and Theorem \ref{thm1}, we have that 
       		\begin{eqnarray*}
       		    \label{l25}
       		    \mathbb{P}( \sqrt{n}(\hat d_{1} - d_{1}) < \sqrt{n}(\hat{q}_{\alpha,1}^* - \hat{\hat d}_1)) \xrightarrow{n \to \infty} \alpha.
       		\end{eqnarray*}
       		Since $ \hat{\hat d}_{1} > \epsilon$ implies $\hat d_{1} - d_{1} > 0$ and \eqref{l16} holds, we obtain
       		\begin{eqnarray*}
       		\label{l26}
       		\mathbb{P}( \hat d_{1} - d_{1} < \hat{q}_{\alpha,1}^* - \hat{\hat d}_{1} \ , \hat{\hat d}_{1} > \epsilon) \leq \mathbb{P}( 0 < \hat{q}_{\alpha,1}^* - \hat{\hat d}_{1} ) \xrightarrow{n \to \infty} 0,
       		\end{eqnarray*}
       		which completes the proof of a).
     \smallskip

    {\bf Proof of b).}
       		The result follows by the same arguments as given for the proof of the second statement of Theorem 2 in \cite{dette2018equivalence}. Only note that the map  $ 		 (b_{1},b_{2}) \mapsto d_{1}(b_{1},b_{2}) $
         from $
         B_{1} \times B_{2}$  onto  $\mathbb{R} $
       		is uniformly continuous, since it is a continuous function on a compact set.

       	\subsection{Proof of the statement 
        in Remark \ref{r1}(b)}

      Consider first the null hypothesis $H_0: d_1 \geq \epsilon$. 
       			  From Lemma \ref{b2} and Theorem 3.2 in \cite{fang2019inference} we know that conditionally in probability
       			\begin{eqnarray*}
       			\label{l27}
       			    \sqrt{n} \ts \hat{\Phi}^{'*} :=  \hat{\Phi}^{'}(\{\sqrt{n}( \hat{\theta}^{*}-\hat{\theta})\}_{x \in \mathcal{X}}) \xrightarrow{d} \Phi^{'}_{ \theta }( \{G(x)\}_{x \in \mathcal{X}} ) = T.
       			\end{eqnarray*}
       			\textit{\underbar{Case 1:  $d_{1} > \epsilon$.}} First note that for all $\alpha \in (0,1) $ we have
       			\begin{eqnarray}
       			\label{l28}
       			\sqrt{n} \ \hat{q}_{\alpha,1}^* \xrightarrow{\mathbb{P}} q_{\alpha},
       			\end{eqnarray}
       			where $q_{\alpha}$ denotes the $\alpha$-quantile of $T$ and $\hat{q}_{\alpha,1}^*$ denotes the $\alpha$-quantile of $\hat{\Phi}^{'*}$. To see this, observe that by definition of $\hat{q}_{\alpha,1}^*$ we have
       			\begin{eqnarray*}
       			\alpha &=& \mathbb{P}( \hat{\Phi}^{'*} < \hat{q}_{\alpha,1}^* \mid \mathcal{Y}  ) = \mathbb{P}( \sqrt{n} \ts \hat{\Phi}^{'*} < \sqrt{n} \ts \hat{q}_{\alpha,1}^* \mid \mathcal{Y}  ) \ \text{ a.s. } \label{l29}
       			\end{eqnarray*}
       			Since $\sqrt{n} \ts \hat{\Phi}^{'*}$ converges in distribution to $T$ conditionally on $\mathcal{Y}$ in probability, Lemma 21.2 in \cite{van2000asymptotic} yields the result \eqref{l28}. Using \eqref{l28} we then obtain that $ \sqrt{n} ( \hat{q}_{\alpha,1}^* + \epsilon - d_{1} ) \xrightarrow{n \to \infty} -\infty $, since $d_{1} > \epsilon$. Combining this result with the fact that $ \sqrt{n}(\hat d_{1}-d_{1}) $ converges in distribution by Theorem \ref{thm1}, we can conclude that
       			\begin{eqnarray*}
       			\label{l30}
       			\mathbb{P}(\hat d_{1} < \hat{q}_{\alpha,1}^* + \epsilon) &= \mathbb{P}(\sqrt{n}(\hat d_{1}-d_{1}) < \sqrt{n}(\hat{q}_{\alpha,1}^* + \epsilon - d_{1})) \xrightarrow{n \to \infty} 0.
       			\end{eqnarray*}

       			\textit{\underbar{Case 2: $d_{1} = \epsilon$}.} Since $ \sqrt{n}(\hat d_{1}-d_{1}) $ converges in distribution to $T$ and \eqref{l28} holds, we deduce that 
       			\begin{eqnarray*}
       			\label{l31}
       			\mathbb{P}(\hat d_{1} < \hat{q}_{\alpha,1}^* + \epsilon) &=& \mathbb{P}(\sqrt{n}(\hat d_{1}-d_{1}) < \sqrt{n}(\hat{q}_{\alpha,1}^* + \epsilon - d_{1})) \\
          &=& \mathbb{P}(\sqrt{n}(\hat d_{1}-d_{1}) < \sqrt{n} \ts \hat{q}_{\alpha,1}^* ) \xrightarrow{n \to \infty} \alpha.
       			\end{eqnarray*}
       			
  Next we consider the alternative $H_1: d_1 < \epsilon$. 
       			Using \eqref{l28} and $d_{1} < \epsilon$, we deduce that $\sqrt{n}(\hat{q}_{\alpha,1}^* + \epsilon - d_{1}) \xrightarrow{n \to \infty} \infty$. Since $ \sqrt{n}(\hat d_{1}-d_{1}) $ converges in distribution, this implies that
       			\begin{eqnarray*}
       			\label{l32}
       			\mathbb{P}(\hat d_{1} < \hat{q}_{\alpha,1}^* + \epsilon) &= \mathbb{P}(\sqrt{n}(\hat d_{1}-d_{1}) < \sqrt{n}(\hat{q}_{\alpha,1}^* + \epsilon - d_{1})) \xrightarrow{n \to \infty} 1.
       			\end{eqnarray*}

        	\begin{Lem}
        	     \label{b2}
        	 {\it 
        	     The sequence of functions
       	    \begin{eqnarray*}
       	    \label{l33}
       	        \hat{\Phi}^{'}(h) \coloneqq \int_{ |{\hat{\theta}} | \geq  {1}/ s_{n}} \operatorname{sgn}(\hat{\theta}(x)) \ts h(x) \ dx \ + \int_{ |{\hat{\theta}} | <  {1}/ s_{n}}  \abs{h(x)} \ dx
       	    \end{eqnarray*}
       	    with $s_{n}/\sqrt{n} \to 0$ satisfies Assumption 4 in \cite{fang2019inference}, meaning for $h \in \l^{\infty}(\mathcal{X})$ we have
       	    \begin{eqnarray*}
       	    \label{l34}
       	        \big  | { \hat{\Phi}^{'}(h) - \Phi_{\theta}^{'}(h) }   \big  | \xrightarrow{\mathbb{P}} 0. 
       	    \end{eqnarray*}
       	    (note that since $\hat{\Phi}^{'}$ is Lipschitz continuous with respect to $\norm{\cdot}_{1}$, it suffices to prove this simpler condition; see \cite{fang2019inference}).
            }
        	 \end{Lem}
       	    \begin{proof}
       	        Defining 
       	        \begin{align*}
       	            A &\coloneqq \Big | { \int_{ | {\hat{\theta}}|  \geq  {1}/{s_{n}} } \operatorname{sgn}(\hat{\theta}(x)) \ts h(x) \ dx \ - \int_{ \abs{ \theta } > 0 } \operatorname{sgn}(\theta(x)) \ts h(x) \ dx  } \Big |  \\
       	            B &\coloneqq   \Big  | { \int_{   |  {\hat{\theta}}| <  {1}/{s_{n}} } \abs{h(x)} \ dx - \int_{ \abs{\theta} = 0 } \abs{h(x)} \ dx \ }
                    \Big | 
       	        \end{align*}
            we note that $\big  | { \hat{\Phi}^{'}(h) - \Phi_{\theta}^{'}(h) } \big  | \leq A + B$ by the triangle inequality. Therefore, it suffices to show that $A \xrightarrow{\mathbb{P}} 0$ and $B \xrightarrow{\mathbb{P}} 0$. In order to show the former (the latter can be proven by similar arguments), we define the sets
            $$ M_{1} \coloneqq  \Big\{  | \hat{\theta} | > \dfrac{1}{s_{n}} \Big\}, \quad M_{2} \coloneqq \Big\{ \abs{\theta} > 0 \Big\} $$ and note that 
            \begin{eqnarray}
                A \leq & \ts \ts \Big | 
                { \int _{ M_{1} \cap M_{2}^{c} } \operatorname{sgn}(\hat{\theta}(x)) \ts h(x) \ dx \ - \int_{ M_{2} \cap M_{1}^{c} } \operatorname{sgn}(\theta(x)) \ts h(x) \ dx  } 
                \Big |  \nonumber \\
                & \quad + \Big |  {  \int_{ M_{1} \cap M_{2} } \Big( \operatorname{sgn}(\hat{\theta}(x)) - \operatorname{sgn}(\theta(x)) \Big) \ts h(x) \ dx }  \Big | \nonumber   \\
                \leq& \ts \ts \lambda( M_{1} \cap M_{2}^{c} ) \norm{h}_{\infty} + \lambda(M_{2} \cap M_{1}^{c}) \ts \norm{h}_{\infty} + o_{\mathbb{P}}(1). \label{l36}
            \end{eqnarray}
            due to $\hat{\theta}\xrightarrow{\mathbb{P}} \theta$ and where $\lambda$ denotes the Lebesgue measure. Therefore, it suffices to show that the first two summands in \eqref{l36} converge to zero in probability. Regarding the first term, we have that 
            \begin{align*}
                \lambda\Big( M_{1} \cap M_{2}^{c} \Big) \ts  &= \lambda\Big( | {\hat{\theta}} | > \frac{1}{s_{n}} , \ts \theta = 0  \Big)  = \lambda\Big( s_{n} | {\hat{\theta} - \theta}|  > 1, \ts \theta = 0 \Big) \\
                &\leq \lambda\Big(s_{n} | { \hat{\theta} - \theta }|  > 1 \Big) = \lambda\Big(  \frac{s_{n}}{\sqrt{n}} \sqrt{n} | { \hat{\theta} - \theta }|  > 1 \Big),
            \end{align*}
            where the last term converges to zero, since $s_{n}/\sqrt{n} \to 0$ by assumption and since the sequence $\sqrt{n} | { \hat{\theta} - \theta }| $ is tight. The second summand can be handled similarly.
            
        	\end{proof}

\bigskip\noindent
\textbf{Acknowledgements:}  
%  The authors would like to thank two unknown referees for their constructive comments on an earlier version of this paper.
This research is supported by the European Union through the European Joint Programme on Rare Diseases under the European Union's Horizon 2020 Research and Innovation Programme Grant Agreement Number 825575.        

       	    \bibliographystyle{apalike}
           	\bibliography{bibliography}

\end{document}